\documentclass{amsart}

\usepackage{babel}
\usepackage{smfthm}
\usepackage{bbm} 
\usepackage{verbatim} 
\usepackage{comment}
\usepackage[hidelinks]{hyperref}

\theoremstyle{plain}
\newtheorem*{Theorem*}{Theorem} 
\newtheorem{Theorem}{Theorem}
\newtheorem{Lemma}{Lemma}
\newtheorem{Corollary}{Corollary}
\newtheorem{Proposition}{Proposition}

\theoremstyle{remark}

\newtheorem*{Claim*}{Claim}

\newcommand{\R}{\mathbb{R}} 
\newcommand{\V}{\mathbf{V}} 
\newcommand{\bN}{\mathbf{N}}
\newcommand{\diff}{\mathop{}\mathopen{}\mathrm{d}} 
\newcommand{\ind}{\mathbbm{1}} 
\newcommand{\hel} { 
\hskip2.5pt{\vrule height7pt width.5pt depth0pt}
\hskip-.2pt\vbox{\hrule height.5pt width7pt depth0pt}
\, }

\newcommand{\calC}{\mathscr{C}}
\newcommand{\calF}{\mathscr{F}}
\newcommand{\calH}{\mathscr{H}}

\newcommand{\calL}{\mathscr{L}}
\newcommand{\calP}{\mathscr{P}}
\newcommand{\calS}{\mathscr{S}}
\newcommand{\bomega}{\boldsymbol{\omega}}
\newcommand{\btheta}{\boldsymbol{\theta}}

\newcommand{\rmch}{\operatorname{ch}}
\newcommand{\cless}{\operatorname{cl}_\mathrm{e}}
\newcommand{\rmsch}{\operatorname{sch}}
\DeclareMathOperator{\rmint}{\mathrm{int}} 
\DeclareMathOperator{\rmLip}{\operatorname{Lip}} 

\renewcommand{\leq}{\leqslant}
\renewcommand{\geq}{\geqslant}
\renewcommand{\subset}{\subseteq}

\usepackage{newtxtext}
\usepackage{newtxmath}
\DeclareMathAlphabet\euscr{U}{eus}{m}{n}
\newcommand{\niceBV}{\euscr{BV}}
\newcommand{\diam}{\operatorname{diam}}
\newcommand{\reg}{\operatorname{reg}}
\newcommand{\symdif}{\mathbin{\triangle}}

\author[Ph. Bouafia]{Philippe Bouafia}

\address{F\'ed\'eration de Math\'ematiques FR3487 \\
  CentraleSup\'elec \\
  3 rue Joliot Curie \\
  91190 Gif-sur-Yvette
}

\email{philippe.bouafia@centralesupelec.fr}

\title{Young integration with respect to Hölder charges}

\begin{document}

\begin{abstract}
  We present a multidimensional Young integral that enables to
  integrate Hölder continuous functions with respect to a Hölder
  charge. It encompasses the integration of Hölder differential forms
  introduced by R. Züst: if $f, g_1, \dots, g_d$ are merely Hölder
  continuous functions on the cube $[0, 1]^d$ whose Hölder exponents
  satisfy a certain condition, it is possible to interpret
  $\diff g_1 \wedge \cdots \wedge \diff g_d$ as a Hölder
  charge and thus to make sense of the integral
  \[
  \int_{B} f \diff g_1 \wedge \cdots \wedge \diff g_d
  \]
  over a set $B \subset [0, 1]^d$ of finite perimeter.
\end{abstract}



\maketitle{}
\setcounter{tocdepth}{1}
\tableofcontents


\section{Introduction}

It is well-known that the Riemann-Stieltjes integral $\int_{a}^b f
\diff g$ converges when $f$ is a continuous function and $g$ is a
function of bounded variation. In a notable work by L. C. Young
\cite{Youn}, it was shown this integral converges in a case where the
regularity condition on the integrator $g$ is, to some extent,
relaxed, and the condition on the integrand $f$ is
strengthened. Specifically, when $f$ is $\beta$-Hölder continuous, $g$
is $\gamma$-Hölder continuous, and $\beta + \gamma > 1$. More
recently, a multidimensional extension of the Young integral was
discovered by R. Züst \cite{Zust}, which gives meaning to integrals of
the form:
\[
\int_{[0, 1]^d} f \diff g_1 \wedge \cdots \wedge \diff g_d
\]
provided that $f$ is $\beta$-Hölder continuous, the $g_i$ are
$\gamma_i$-Hölder continuous, and
\begin{equation}
\label{eq:zustcondition}
\beta + \gamma_1 + \cdots + \gamma_d > d.
\end{equation}
This paper aims to offer a generalization of these results in cases
where the integrator is a Hölder charge. We will dedicate a large part
of this introduction to explaining the nature of this object.

\subsection{Charges and Hölder regularity}
The notion of charge originates from non-absolute integration
theories, in the work of Ma\v{r}\'ik \cite{KartMari}. It has then been
successfully applied in the construction of well-behaved integrals,
such as the Pfeffer integral \cite{Pfef3}.

A charge on $[0, 1]^d$ is a function $\bomega \colon \niceBV([0, 1]^d)
\to \R$ defined on the algebra $\niceBV([0, 1]^d)$ of sets of finite
perimeter in $[0, 1]^d$ (referred to as $BV$-sets in this article). It
is finitely additive and satisfies a certain continuity condition. By
integration with respect to a charge, we typically mean an integral
based on Riemann-Stieltjes approximations
\[
\int_{[0, 1]^d} f \, \bomega \simeq \sum_{(B, x) \in \calP} f(x) \bomega(B)
\]
where $\calP$ is a tagged partition of $[0, 1]^d$ into $BV$-sets. The
class of $BV$-sets strikes a good balance between geometric rigidity
(invariance under biLipschitz transformations) and generality, as it
satisfies compactness properties, and its use leads to the
construction of integrals with good properties, such as \cite{Pfef4}.

In dimension $d = 1$, it can be shown that charges are differentials
$\diff \varv$ of continuous functions $\varv$ on $[0 ,1]$. The
$BV$-subsets of $[0, 1]$ have a particularly simple structure; they
are finite unions of compact intervals, up to negligible sets. The
charge $\diff \varv$ measures the increments of $\varv$:
\[
\diff \varv \left( \bigcup_{i=1}^p [a_i, b_i] \right) = \sum_{i=1}^p
\left( \varv(b_i) - \varv(a_i) \right).
\]
In \cite{BouaDePa}, a notion of Hölder regularity for charges was
introduced, which we delve into further in this paper. In dimension
$d=1$, it is natural to require that a charge $\diff \varv$ is
$\gamma$-Hölder when $\varv$ is, that is to say
\[
\lvert \diff \varv(I) \rvert \leq C \lvert I \rvert^\gamma
\]
when $I \subset [0, 1]$ is an interval and $\lvert I \rvert$ its
Lebesgue measure. This inequality becomes invalid when we substitute
$I$ with an arbitrary $BV$-set $B = \bigcup_{i=1}^p [a_i, b_i]$. In
such cases, we need to account for the regularity of $B$, which, in
dimension one, is simply measured by the number $p$ of connected
components of $B$, or, equivalently, by the perimeter $\|B\| =
2p$. Indeed, a straightforward application of Hölder's inequality
gives
\[
\left\lvert\diff \varv(B) \right\rvert \leq C \sum_{i=1}^p \lvert b_i
- a_i \rvert^{\gamma} \leq C p^{1- \gamma} \lvert B \rvert^\gamma.
\]

It is only from dimension $d \geq 2$ that the concepts of charges and
Hölder regularity of charges become genuinely new, with no counterpart
in terms of functions. A charge $\bomega$ over $[0, 1]^d$ is regarded
as $\gamma$-Hölder (for $(d-1)/d < \gamma < 1$) if it satisfies the
condition that there is a constant $C \geq 0$ such that
\[
\lvert \bomega(K) \rvert \leq C \lvert K \rvert^\gamma
\]
for any dyadic cube $K$.  Theorem~\ref{thm:holdintr} provides an
intrinsic characterization of the Hölder behavior of charges,
generalizing the elementary reasoning from the previous paragraph to
higher dimensions. In broad terms, this theorem asserts that the
characteristic property of a $\gamma$-Hölder charge $\bomega$ is that
its values on $BV$-sets $B$ are bounded by $\lvert B \rvert^\gamma$,
up to a multiplicative ``constant'' dependent on the isoperimetric
coefficient of $B$:
\[
\lvert \bomega(B) \rvert \leq \frac{C}{(\operatorname{isop} B)^{d(1 -
    \gamma)}} \lvert B \rvert^\gamma, \qquad \text{where }
\operatorname{isop} B = \frac{\lvert B \rvert^{\frac{d-1}{d}}}{\|B\|}.
\]

Interesting examples of Hölder charges can be constructed by looking
at the sample paths of the fractional Brownian sheet $\{W_{t_1, \dots,
  t_d} : 0 \leq t_1, \dots, t_d \leq 1\}$ with Hurst parameters $(H_1,
\dots, H_d)$ satisfying:
\[
\bar{H} = \frac{H_1 + \cdots + H_d}{d} > \frac{d-1}{d}.
\]
Let us define the increment of $W$ over a rectangle $K = \prod_{i=1}^d
[a_i, b_i]$ by the formula
\[
\Delta_{W} \left( \prod_{i=1}^d [a_i, b_i] \right) = \sum_{(c_i) \in
  \prod_{i=1}^d \{a_i, b_i\}} (-1)^{\delta_{a_1, c_1}} \ldots
(-1)^{\delta_{a_d, c_d}} W_{c_1, \dots, c_d}
\]
where, for all real numbers $t, t'$, the number $\delta_{t, t'}$ is 1
or 0 depending on whether $t = t'$ or not. (In dimension $d=2$, we
recover the well-known rectangular increment formula $\Delta_W([a_1,
  b_1] \times [a_2, b_2]) = W_{b_1, b_2} - W_{a_1, b_2} - W_{b_1, a_2}
+ W_{a_1, a_2}$). By finite additivity, it is possible to extend the
premeasure $\Delta_{W}$ to the class of rectangular figures (finite
unions of rectangles), which is dense in $\niceBV([0, 1]^d)$. The
``chargeability'' theorem \cite[Theorem~9.1]{BouaDePa} asserts that,
almost surely, $\Delta_W$ can be further extended to $\niceBV([0,
  1]^d)$, and that this extension, still denoted $\Delta_W$, is a
$\gamma$-Hölder charge for all $\frac{d-1}{d} < \gamma <
\bar{H}$. Thus, by means of the (random) charge $\Delta_W$, it is now
possible to consider the increments of the sample paths of $W$ over
arbitrary $BV$-sets, a property that is called chargeability,
see~\cite[3.4]{BouaDePa}.

In reality, this Hölder behavior could have been vaguely guessed, at
least in mean, by examining the standard deviation of the increments
of the fractional Brownian sheet over rectangles:
\[
\mathbb{E} \left(\Delta_{W}(K)^2 \right)^{1/2} = \prod_{i=1}^d |b_i -
a_i|^{H_i}, \qquad \text{for } K = \prod_{i=1}^d [a_i, b_i].
\]
This standard deviation is approximately equal to $\lvert K
\rvert^{\bar{H}}$, if $K$ is close to being a cube, meaning that the
isoperimetric coefficient of $K$ is high.

In this article, we provide another class of examples of Hölder
charges. If $g_1, \dots, g_d: [0, 1]^d \to \R$ are Hölder functions
with exponents $\gamma_1, \dots, \gamma_d$ satisfying $\gamma_1 +
\cdots + \gamma_d > d-1$, it is possible to associate them with a
Hölder charge denoted as $\diff g_1 \wedge \cdots \wedge \diff g_d$
with exponent $(\gamma_1 + \cdots + \gamma_d)/d$.

It is worth mentioning that, despite their apparent similarity to
measures, charges (more precisely, strong charges,
see~\ref{e:strongCharges}) have a much richer structure than
expected. They form a Banach space, which is predual to $BV$
(see~\cite[Theorem~4.2]{BouaDePa}). The article~\cite{DePaMoonPfef}
extends the notion of strong charge in arbitrary codimensions and
defines an exterior differentiation operator. In a nutshell, charges
are generalized differential forms, which explains the notation
$\bomega$ we use for charges. They even give rise to a charge
cohomology, similar to the de Rham cohomology, which has been studied
in \cite{DePaHardPfef}.

\subsection{Young integral}
There is a purely analytical definition of the $1$-dimensional Young
integral $\int_0^1 f \diff g$, which involves bases of function
spaces. More precisely, we decompose the integrand $f$ into the Haar
system, consisting of the constant function $h_{-1} = 1$ and the
square-shaped functions $h_{n,k}$ (for $n \geq 0$ and $0 \leq k \leq
2^{n-1}$):
\[
f(x) = \alpha_{-1} h_{-1}(x) + \sum_{n=0}^\infty \sum_{k=0}^{2^n-1}
\alpha_{n,k} h_{n,k}(x)
\]
and we decompose the integrator $g$, that is assumed to vanish at $0$,
along the Faber-Schauder system, consisting of the primitives
$f_{-1}$, $f_{n,k}$ of the Haar functions:
\[
g(x)  = \beta_{-1} f_{-1}(x) +
\sum_{n=0}^\infty \sum_{k=0}^{2^n-1} \beta_{n,k} f_{n,k}(x).
\]
Then we define
\[
\int_{0}^1 f \diff g = \alpha_{-1}\beta_{-1} + \sum_{n=0}^\infty
\sum_{k=0}^{2^n-1} \alpha_{n,k} \beta_{n,k}
\]
The above series converges provided that $f$ and $g$ are regular
enough. The initial difficulty in generalizing this construction to
dimensions $d \geq 2$ is that the multidimensional Faber-Schauder
system, in the space of continuous functions of more than one
variable, loses its connection with the Haar system, and furthermore,
it is no longer a wavelet basis. If, however, we allow the integrator
to be a (strong) charge, then it is possible to use the
Faber-Schauder-like basis discovered in \cite{BouaDePa}, whose
elements are obtained by antidifferentiating the multidimensional Haar
functions in the space of (strong) charges.  The Young integral we
then construct allows the integration of a function $f$ that is
$\beta$-Hölder continuous with respect to a $\gamma$-Hölder charge
$\bomega$, provided that the exponents satisfy $\beta + d\gamma >
d$. Furthermore, it is even possible to localize the construction and
give meaning to $\int_B f \, \bomega$ for any domain $B \in
\niceBV([0, 1]^d)$. The function $B \mapsto \int_B f \, \bomega$,
which we denote simply as $f \cdot \bomega$, is itself a
$\gamma$-Hölder charge.

Towards the end of Section~\ref{sec:young}, we provide an alternative
construction of the Young integral, which is based on a variation of
the sewing lemma, adapted to charges. However, this approach does not
bypass the Faber-Schauder basis, as the proof of the charge sewing
lemma indeed relies on an analytical lemma
(Lemma~\ref{prop:gammaHolder=>strong}) for which we currently lack a
method of proof without using the Faber-Schauder basis.

A deeper connection with Riemann-Stieltjes sums is established at a
later stage, in Section~\ref{sec:pfeffer}, where we show that the
Young integral defined here is a special case of the Pfeffer integral
(with respect to a charge, known as the generalized Riemann integral
in \cite{Pfef4}), just as the one-dimensional Young integral is an
instance of the Riemann-Stieltjes integral. The space of
Pfeffer-integrable functions is extremely large but poorly
understood. However, it is known to contains most functions of bounded
variation, see \cite[Theorem~3.4 and Proposition~3.6]{Pfef4}.

When we employ the Young integral in the context of the two examples
mentioned earlier, it offers the following capabilities:
\begin{itemize}
\item a $\beta$-Hölder continuous function can be integrated with
  respect to $\mathrm{d} g_1 \wedge \cdots \wedge \diff g_d$,
  if~\eqref{eq:zustcondition} is satisfied. This coincides with the
  results from Züst.
\item it allows to integrate a process $X$, indexed over $[0, 1]^d$,
  with $\beta$-Hölder continuous sample paths, with respect to the
  variations of the fractional Brownian sheet (that is, with respect
  to $\Delta_W$), provided that
  \[
  \beta + H_1 + \cdots + H_d > d.
  \]
  Such an integral is understood in a pathwise sense.
\end{itemize}

A natural follow-up question is whether the integral we defined can be
extended beyond the Young regime, potentially by using ideas from
rough paths theory. A weak signal indicating that this might be
achievable is the fact that our integral is iterable, with a
well-defined \emph{indefinite} integral
(Theorem~\ref{thm:youngIndefinite}), an improvement over Züst's
integration of Hölder forms. The Haar and Faber-Schauder systems could
exploited to develop such an extension, taking ideas from
\cite{GubiImkePerk}.

If this extension exists, it might provide a pathway to solving
the Gromov-Hölder equivalence problem for the first Heisenberg group
$\mathbb{H}$. Indeed, R. Züst demonstrated in
\cite[Theorem~1.3]{Zust3}, through the integration of Hölder
differential forms, that there is no embedding from an nonempty open
subset of $\R^2$ to $\mathbb{H}$ of Hölder regularity $> 2/3$. It is
still an open problem whether the same conclusion holds for Hölder
exponents $> 1/2$. An adaptation of Züst's proof appears to require
the extension we mentioned, see in particular
\cite[Section~4.1]{Zust3}.

The mathematical literature also contains results on multidimensional
extensions of the Young integral; see for example \cite{QuerTind},
\cite{Hara} and \cite{ChouGubi}. While these theories are
better-suited for addressing Young differential equations, they do not
incorporate Züst's integration of Hölder differential forms.

\subsection{Duality results and characterization of fractional Sobolev regularity}
Naturally, one might wonder about the kind of weak geometric objects
upon which Hölder charges can act. This is the question we investigate
in Section~\ref{sec:BValpha}. We exhibit a duality between the space
$\rmsch^\gamma([0, 1]^d)$ of Hölder charges and the fractional Sobolev
space $W^{1-\alpha,1}((0, 1)^d)$ (for $\alpha = d\gamma - (d-1)$). In
fact, our methods give several characterizations of the fractional
Sobolev regularity, among them:
\begin{itemize}
  \item We establish that the space of functions with bounded Züst
    fractional variation, introduced in \cite{Zust2} coincides with
    $W^{1-\alpha,1}_c(\R^d)$. This finding answers a question posed in
    \cite{ComiStef} concerning the comparison between the classes of
    functions with bounded fractional variation in the senses of Züst
    and Comi-Stefani, with the latter being strictly larger.
  \item We prove that a compactly supported function $f \in L^1(\R^d)$
    is in $W^{s, 1}_c(\R^d)$ (for $0 < s < 1$) if and only if
    \[
    \sup \left\{ \int f \operatorname{div} \varv : \varv \in C^1(\R^d;
    \R^d) \text{ and } \rmLip^{1-s} \varv \leq 1 \right\} < \infty
    \]
    (where $\rmLip^{1-s} \varv$ is the $(1-s)$-Hölder exponent of
    $\varv$.)
\end{itemize}

The paper is organized into seven sections with self-explanatory
titles. Sections~\ref{sec:holderDF} and~\ref{sec:BValpha} can be read
independently from Section~\ref{sec:pfeffer}.

\subsection{Acknowledgments}
I would like to express my sincere gratitude to Roger Züst for his
discussions and insightful suggestions.

\section{Preliminaries and Notations}
\label{sec:prelim}

\subsection{General notations}
Throughout this paper, $\R$ will denote the
set of all real numbers. We will work in an ambient space whose
dimension is an integer $d \geq 1$, typically $[0, 1]^d$. The
Euclidean norm of $x \in \R^d$ is denoted $|x|$.

The topological boundary and the interior of a set $E \subset \R^d$
will be denoted $\partial E$ and $\rmint E$, respectively. The
indicator function of $E$ is $\ind_E$. The symmetric difference of two
sets $E_1, E_2 \subset \R^d$ is written $E_1 \symdif E_2$.

Unless otherwise specified, the expressions ``measurable'', ``almost
all'', ``almost everywhere'' tacitly refer to the Lebesgue
measure. The Lebesgue (outer) measure of a set $E \subset \R^d$ is
simply written $|E|$. Two subsets $E_1, E_2 \subset \R^d$ are said to
be almost disjoint whenever $|E_1 \cap E_2| = 0$. If $U \subset \R^d$
is a measurable set and $1 \leq p \leq \infty$, the Lebesgue spaces
are denoted $L^p(U)$. Here, $U$ is endowed with its Lebesgue
$\sigma$-algebra and the Lebesgue measure. The $L^p$ norm is written
$\|\cdot\|_p$. The notation $\|\cdot\|_\infty$ might also refer to the
supremum norm in the space of continuous functions. The integral of a
function $f$ with respect to the Lebesgue measure is simply written
$\int f$, with no mention of the Lebesgue measure. In case another
measure is used, it will be clear from the notation.

The dual of a Banach space $X$ is $X^*$. Unless otherwise specified,
the operator norm of a continuous linear map $T$ between normed spaces
will be written $\|T\|$. Throughout the paper, the letter $C$
generally refers to a constant, that may vary from line to line
(sometimes even on the same line).

\subsection{$BV$ functions} The variation of a Lebesgue integrable function
$u \colon \R^d \to \R$ over an open subset $U \subset \R^d$ is the
quantity
\[
\V u = \sup \left\{ \int u \operatorname{div} \varv : \varv \in
C^1_c(\R^d; \R^d) \text{ and } \lvert\varv(x)\rvert \leq 1 \text{ for
  all } x \in \R^d \right\}
\]
where $C^1_c(\R^d; \R^d)$ denotes the space of continuously
differentiable compactly supported vector fields on $\R^d$.

The function $u$ is said to be \emph{of bounded variation} whenever
$\V u < \infty$. The set of (equivalence classes of) integrable
functions of bounded variation on $\R^d$ is denoted $BV(\R^d)$. It is
a Banach space under the norm $u \mapsto \|u\|_1 + \V u$.
  
\subsection{$BV$-sets}
\label{e:BVsets}
The \emph{perimeter} of a measurable subset $A$ of $\R^d$ is the
extended real number defined by $\|A\| = \V \ind_A$. We will say that
$A$ is a \emph{$BV$-set} whenever $A$ is bounded, measurable and
$\|A\| < \infty$.

For a fixed $BV$-set $A \subset \R^d$, we introduce the collection
$\niceBV(A)$ of $BV$-subsets of $A$. It can be proven that
$\niceBV(A)$ is an algebra of sets, as, for all $B, B_1, B_2 \in
\niceBV(A)$, one has
\begin{equation}
  \label{eq:perimeters}
  \|A \setminus B\| \leq \|A\| + \|B\| \text{ and } \max(\| B_1 \cup
  B_2 \|, \|B_1 \cap B_2\|) \leq \|B_1\| + \|B_2\|
\end{equation}
see \cite[Proposition~1.8.4]{Pfef2}.  In order to quantify the
roundedness of a $BV$-set $A$, we introduce its \emph{isoperimetric
  coefficient}
\[
\operatorname{isop} A = \frac{\lvert A\rvert^{\frac{d-1}{d}}}{\|A\|}.
\]

Let $A$ be a $BV$-set of $\R^d$. We equip the collection $\niceBV(A)$
with the following notion of convergence: a sequence $(B_n)$
\emph{$BV$-converges} to $B$ whenever
\[
\lvert B_n \symdif B \rvert \to 0 \text{ and } \sup_{n \geq 0}
\|B_n\| < \infty
\]
This notion of convergence is compatible with a certain topology, that
is described in \cite[Section 1.8]{Pfef2} (for a slightly different
notion of charges, that are defined over all $\R^d$), but we shall not
use it. Note that the limit of a $BV$-converging sequence is unique
only up to negligible subsets of $A$. Using~\eqref{eq:perimeters}, it
may be proven that the set-theoretic operations are continuous with
respect to $BV$-convergence.
  
Among the subsets of $[0, 1]^d$, we may consider some that are more
regular that $BV$-sets, which we will describe here.  A \emph{dyadic
  cube} is a set of the form $\prod_{i=1}^d \left[2^{-n} k_i, 2^{-n}
  (k_i+1) \right]$, where $n \geq 0$ and $0 \leq k_1, \dots, k_d \leq
2^n-1$ are integers. Such a dyadic cube has side length $2^{-n}$ and
we will say that it is \emph{of generation} $n$. A \emph{dyadic
  figure} is a set that can be written as a finite union of dyadic
cubes. In particular, the empty set is a dyadic figure. The collection
of dyadic cubes and dyadic figures are denoted
$\calC_{\mathrm{dyadic}}([0, 1]^d)$ and $\calF_{\mathrm{dyadic}}([0,
  1]^d)$, respectively.

\subsection{Charges} A \emph{charge} over $A$ is a function $\bomega \colon
\niceBV(A) \to \R$ that satisfies the properties
\begin{itemize}
\item[(A)] Finite additivity: $\bomega(B_1 \cup B_2) = \bomega(B_1) +
  \bomega(B_2)$ for any almost disjoint $BV$-subsets $B_1, B_2$ of $A$;
\item[(B)] Continuity: if $(B_n)$ is a sequence in $\niceBV(A)$ that
  $BV$-converges to some $B \in \niceBV(A)$, then $\bomega(B_n) \to
  \bomega(B)$.
\end{itemize}
Throughout the article, we will use the easy observations:
\begin{itemize}
\item $\bomega(B)$ depends only on the equivalence class of $B$. In
  particular, a charge vanishes on the negligible subsets of $A$;
\item to check (B), it is enough to show that $\bomega(B_n) \to 0$ for
  sequences $(B_n)$ in $\niceBV(A)$ that $BV$-converge to
  $\emptyset$.
\end{itemize}
The linear space of charges over $A$ is denoted $\rmch(A)$.

Charges over $[0, 1]$ are particularly simple to describe. This is
because the $BV$-subsets of $[0, 1]$ are just the finite unions of
compact intervals, up to negligible sets. To each continuous function
$\varv$ on $[0, 1]$, we associate the charge $\diff \varv$, defined by
\begin{equation}
  \label{eq:chargedim1}
  \diff \varv \colon \bigcup_{i=1}^p [a_i, b_i] \mapsto \sum_{i=1}^p
  \left( \varv(b_i) - \varv(a_i) \right).
\end{equation}
Property (B) for $\diff \varv$ follows from the continuity of
$\varv$. It is clear that $\diff\varv = \diff(\varv - \varv(0))$; thus
we can always assume that $\varv$ vanishes at $0$.  Conversely, we
derive from any charge $\bomega \in \mathrm{ch}([0, 1])$ a function
$\varv \colon x \mapsto \bomega([0, x])$ that is continuous and
vanishes at $0$. We let the reader check that we defined a one-to-one
correspondence between $\mathrm{ch}([0, 1])$ and $C_0([0, 1])$, the
space of continuous functions vanishing at $0$.

We let $\mathrm{d} \calL$ denote the \emph{Lebesgue charge}, that
sends a $BV$-set $B \in \niceBV(A)$ to its Lebesgue measure $\lvert B
\rvert$. More generally, if $g \in L^1(A)$, we can define the charge
$g \diff \calL$ over $A$ by
\[
(g \diff \calL)(B) = \int_B g, \qquad \text{for } B \in \niceBV(A).
\]
The function $g$ is the \emph{density} of $g \diff \calL$. In a sense,
the function $g$ can be thought of as a derivative of $g \diff \calL$,
and, reciprocally, $g \diff \calL$ as the indefinite integral of
$g$. Informally, the operation $g \mapsto g \diff \calL$ is a way to
antidifferentiate $g$ in the space of charges. In the case $A = [0,
  1]$, the continuous function associated to $g \diff \calL$ is indeed
the primitive of $g$, in the ordinary sense, as
\[
(g \diff \calL)([0, x]) = \int_0^x g
\]
for all $x \in [0, 1]$.

The following theorem, due to De Giorgi and proven in
\cite[Proposition 1.10.3]{Pfef}, asserts that any $BV$-set can be
approximated by dyadic figures. It entails a result that we will use
repeatedly: if a charge on $[0, 1]^d$ vanishes over dyadic cubes, then
it is the zero charge.

\begin{Theorem}
  \label{thm:DeGiorgi}
  For any $B \in \niceBV([0, 1]^d)$, there is a sequence $(B_n)$ of
  dyadic figures that $BV$-converges to $B$.
\end{Theorem}

\subsection{Restriction of charges} Let $A$ be a $BV$-set and $A' \in
\niceBV(A)$. The \emph{restriction} of a charge $\bomega \in
\operatorname{ch}(A)$ to $A'$ is the charge $\bomega_{\mid A'} \in
\operatorname{ch}(A')$ defined by $\bomega_{\mid A'}(B) = \bomega(B)$
for all $B \in \niceBV(A')$.

\subsection{Charges of the form $\bomega \circ \Phi$}
\label{e:pullback}
Let $A, A'$ be two $BV$-subsets of $\R^d$ and $\Phi \colon A \to A'$ a
biLipschitz map from $A$ to $\Phi(A)$. For any $B \in \niceBV(A)$, the
subset $\Phi(B)$ is a $BV$-subset of $A'$, whose volume and perimeter
is estimated by
\[
\lvert \Phi(B) \rvert \leq \left( \rmLip \Phi \right)^d \lvert B
\rvert, \qquad \| \Phi(B) \| \leq \left( \rmLip \Phi
\right)^{d-1} \|B\|.
\]
This shows in particular that, if $(B_n)$ is a sequence in
$\niceBV(A)$ that $BV$-converges to $\emptyset$, then $(\Phi(B_n))$
$BV$-converges to $\emptyset$ in $\niceBV(A')$.  Thus, we can
pull back any charge $\bomega \in \mathrm{ch}(A')$ to the charge
$\bomega \circ \Phi \in \mathrm{ch}(A)$ defined by
\[
(\bomega \circ \Phi) (B) = \bomega(\Phi(B)), \qquad B \in \niceBV(A).
\]

\subsection{The $\hel$ operation} Let $\bomega \in \rmch(A)$ be a charge
over a $BV$-set $A$, and let $A' \in \niceBV(A)$. We define the charge
$\bomega \hel A' \in \rmch(A)$ be
\[
(\bomega \hel A')(B) = \bomega(A' \cap B), \qquad B \in \niceBV(A).
\]
Indeed, $\bomega\hel A'$ is clearly finitely additive, and its
continuity with respect to $BV$-convergence is a consequence of the
formula~\eqref{eq:perimeters}, which guarantees that, for any
sequence $(B_n)$ that $BV$-converges to $B$ in $\niceBV(A)$, the
sequence $(A' \cap B_n)$ $BV$-converges to $A' \cap B$.

\subsection{Strong charges}
\label{e:strongCharges}
Let $A$ be a $BV$-subset of $\R^d$. We let $BV(A)$ the closed linear
subspace of $BV(\R^d)$ that consists of functions $u$ vanishing almost
everywhere outside $A$.

In fact, we shall work with the following norm on $BV(A)$
\[
\|u\|_{BV} = \V u, \qquad \text{for } u \in BV(A).
\]
This norm is equivalent to the more common norm inherited from
$BV(\R^d)$ (see for instance \cite[Paragraph 3.7]{BouaDePa} where this
claim is proven in the case where $A = [0, 1]^d$), but it will prove
more convenient for our purposes.

The Sobolev-Poincaré inequality states that
\[
\|u\|_{d/(d-1)} \leq C \|u\|_{BV}
\]
for any $u \in BV(A)$, where $C$ is a constant (see \cite[5.6.1,
  Theorem 1(iii)]{EvanGari}). In other words, every function in
$BV(A)$ is $L^{d/(d-1)}$ and the inclusion map $BV(A) \to
L^{d/(d-1)}(A)$ (more precisely, the map that sends $u$ to its
restriction $u_{\mid A}$) is continuous.  We call $T \colon L^d(A) \to
BV(A)^*$ the adjoint map. It sends $g \in L^d(A)$ to the linear
functional
\[
T_g \colon u \mapsto \int_A gu
\]
The space $SCH(A)$ of \emph{strong charge functionals} is the closure
of $T(L^d(A))$ in $BV(A)^*$. To each $\alpha \in SCH(A)$ we associate
the map $\calS(\alpha) \colon \niceBV(A) \to \R$ defined by
$\calS(\alpha)(B) = \alpha(\ind_B)$. The space $SCH(A)$ is given the
operator norm $\| \cdot \|$ inherited from $BV(A)^*$ and as such, it
is a Banach space.

We claim that $\calS(\alpha)$ is a charge. The finite additivity of
$\calS(\alpha)$ is clear. As for continuity, it suffices to consider
a sequence $(B_n)$ in $\niceBV(A)$ that $BV$-converges to
$\emptyset$ and prove that $\calS(\alpha)(B_n) \to 0$. Let
$\varepsilon > 0$. As $L^\infty(A)$ is dense in $L^d(A)$, we are
able to pick $g \in L^\infty(A)$ such that $\|\alpha - T_g\| \leq
\varepsilon$. Then
\[
\lvert \calS(\alpha)(B_n) \rvert \leq \left\lvert (\alpha -
T_g)(\ind_{B_n}) \right\rvert + \left\lvert \int_{B_n} g \right\rvert \leq
\varepsilon \|B_n\| + \lvert B_n \rvert \|g\|_{\infty}.
\]
Consequently,
\[
\limsup_{n \to \infty} \calS(\alpha)(B_n) \leq \varepsilon \sup_{n \geq 0} \|B_n\|
\]
and by the arbitrariness of $\varepsilon$, it follows that $\lim
\calS(\alpha)(B_n) = 0$.

Any charge of the form $\calS(\alpha)$ is called \emph{strong}, and
the space of strong charges over $A$ is denoted $\rmsch(A)$.  We claim
that the map
\[
\calS \colon SCH(A) \to \rmsch(A)
\]
is one-to-one, so that strong charges and strong charge functionals
are essentially the same objects. The injectivity of $\calS$ is proven
in \cite[Corollary 4.9]{BouaDePa} in case $A = [0, 1]^d$ (the only
case we shall need). Note that, for $g \in L^d(A)$, the charge $g
\diff \calL$ is strong, as $g \diff \calL = \calS(T_g)$.

There are other ways to think about strong charges and strong charge
functionals, which are perhaps more intuitive. In \cite[Theorem
  4.5]{BouaDePa}, it is proven that strong charge functionals are
characterized as being the distributional divergences of continuous
vector fields (when $A = [0, 1]^d$), a fact we shall not use. In
particular, for $d=1$, all charges on $[0, 1]$ are strong,
by~\eqref{eq:chargedim1}. More precisely, $\varv \mapsto \diff \varv$
is a Banach space isomorphism between $C_0([0, 1])$ and $\rmsch([0,
  1])$.
  
Nevertheless, our approach to introducing strong charges offers two
distinct advantages. First, it enables to equip $\rmsch(A)$ with a
Banach space structure. More specifically, we norm $\rmsch(A)$ in a
such way that $\calS$ becomes an isometry: $\|\bomega\| =
\|\calS^{-1}(\bomega)\|$ for all $\bomega \in \rmsch(A)$. Secondly, it
provides a continuous embedding of $L^d(A)$ into $\rmsch(A)$, realized
through the mapping $\calS \circ T$, which sends $g$ to $g \diff
\calL$. This embedding will prove particularly useful when dealing
with the specific case $A = [0, 1]^d$ to construct a basis of
$\rmsch([0, 1]^d)$, as detailed in~\ref{e:FaberSCH}.

We mention that strong charges possess a much richer structure
compared to ordinary charges, which is a fortunate observation as we
will subsequently demonstrate that Hölder charges
(see~\ref{e:defHolder}) are strong
(Corollary~\ref{cor:holder=>strong}). For instance, strong charges can
be pulled back by Lipschitz functions (this notion of pullback differs
from~\ref{e:pullback} as it takes orientation into account). This is
made possible because strong charges, when viewed as functionals, act
on $BV$ functions. These $BV$ functions are geometric objects; they
are $0$-codimensional normal currents in the sense of Federer-Fleming
\cite{FedeFlem}, that can be pushed forward by Lipschitz maps. By
duality, this defines the pullback operation for strong charges. The
paper \cite{DePaMoonPfef} delves much deeper into this insight,
extending strong charges to arbitrary codimensions, and defining an
exterior derivative. This leads to the notion that strong charges are
akin to generalized differential forms, a concept reflected in our
notations. In contrast, ordinary charges are in duality with
$BV_\infty$ functions (bounded functions of bounded variation) rather
than with $BV$ functions, see \cite[Theorem 4.11]{BuczDePaPfef}. All
of this will not be used in our construction of the Young integral.

\subsection{Haar and Faber-Schauder bases in dimension one} Before
introducing the Faber-Schauder basis of $\rmsch([0, 1]^d)$ let us
first review the definition of the Haar and the Faber-Schauder systems
in dimension one. The Haar system consists of the following functions,
defined on $[0, 1]$:
\[
h_{-1} \colon x \mapsto 1, \qquad h_{n,k} \colon x
\mapsto \begin{cases} 2^{n/2} & \text{if } 2^{-n}k \leq x < 2^{-n}k +
  2^{-(n+1)} \\ -2^{n/2} & \text{if } 2^{-n}k + 2^{-(n+1)} \leq x <
  2^{-n}(k+1)\\ 0 & \text{otherwise}
\end{cases}
\]
for all integers $n \geq 0$ and $k \in \{0, 1, \dots, 2^{nd}-1\}$. It
is well-known that the sequence $h_{-1}, h_{0, 0}, h_{1, 0}, h_{1,1},
\dots$ (indices are ordered lexicographically) is a Schauder basis of
any $L^p([0, 1])$ space ($1 \leq p < \infty$). The normalizing
constant $2^{n/2}$ is chosen so that the Haar functions are
orthonormal in $L^2([0, 1])$; other choices are possible.

The Haar system is in relationship with another system of functions,
the Faber-Schauder system, whose elements are obtained by simply
antidifferentiating the Haar functions:
\[
f_{-1} \colon x \mapsto x, \qquad f_{n,k} \colon \begin{cases}
  2^{n/2}x - 2^{-n/2}k & \text{if } 2^{-n}k \leq x < 2^{-n}k +
  2^{-(n+1)} \\ 2^{-n/2}(k+1) -2^{n/2}x & \text{if } 2^{-n}k + 2^{-(n+1)} \leq x <
  2^{-n}(k+1)\\ 0 & \text{otherwise}
\end{cases}
\]
The Faber-Schauder system $f_{-1}, f_{0, 0}, f_{1,0}, f_{1,1}, \dots$
is a Schauder basis of $C_0([0, 1])$, the space of continuous
functions on $[0, 1]$ vanishing at $0$ (with the supremum norm), as
shown by Schauder~\cite{Schau}.

\subsection{Haar functions and Faber-Schauder system in greater dimension}
\label{e:FaberSCH}
We first introduce the isotropic Haar functions on the cube $[0,
  1]^d$. We let $A_1$ be the Haar matrix
\[
A_1 = \begin{pmatrix} 1 & 1 \\ 1 & -1 \end{pmatrix}
\]
and we construct, by induction on the dimension $d$, a sequence of
matrices $A_1, A_2, \dots$ satisfying
\[
A_{d+1} = \begin{pmatrix} A_d & A_d \\ A_d & -A_d \end{pmatrix} 
\]
The matrix $A_d$ is of order $2^d$, and we shall index its rows and
columns from $0$ to $2^d -1$. By induction on $d$, we can prove that
$A_d$ is a symmetric matrix whose inverse is $2^{-d} A_d$.

Next, we label all the dyadic cubes (in $[0, 1]^d$) of generation $n$
as $K_{n, 0}, \dots, K_{n, 2^{nd}-1}$, in such a way that the $2^d$
children of a cube $K_{n,k}$ (the subcubes of generation $n+1$
contained in $K_{n,k}$) are precisely the $K_{n+1, 2^d k}, K_{n+1, 2^d
  k + 1}, \dots, K_{n+1, 2^d k + 2^d -1}$.

The Haar functions are defined by $g_{-1} = \ind_{[0, 1]^d}$ and
\[
g_{n,k,r} = 2^{nd/2} \sum_{\ell = 0}^{2^d -1} (A_d)_{r, \ell}
\ind_{K_{n+1, 2^d k + \ell}}
\]
for $n \geq 0$, $k\in \{0, \dots, 2^{nd}-1\}$ and $r \in \{1, \dots,
2^d-1\}$.  The indices $n, k, r$ are referred to as the generation
number, the cube number, and the type number. Throughout the article,
we will frequently employ the expression ``for all relevant indices
$n,k,r$'' to mean ``for integers $n \geq 0$, $k\in \{0, \dots,
2^{nd}-1\}$ and $r \in \{1, \dots, 2^d-1\}$''.

Observe that the support of $g_{n,k,r}$ is the cube $K_{n,k}$. Again,
the normalizing constant $2^{nd/2}$ is chosen so that the Haar
functions are orthonormal in $L^2([0, 1]^d)$.

The $2^d -1$ Haar functions $g_{0, 0, r}$ of generation $0$ can be
also be constructed as the tensor products of the form $h_1 \otimes
h_2 \otimes \cdots \otimes h_d$, where the functions $h_1, \dots, h_d$
are $1$-dimensional Haar functions of generation $0$ (that is, either
$h_{0, 0}$ or $h_{0, 1}$), except for the constant function $g_{-1} =
h_{0,0} \otimes \cdots \otimes h_{0, 0}$ that we have intentionally
set aside in the generation $-1$. The Haar functions of the subsequent
generations are, in turn, rescaled, translated and normalized versions
of the $g_{0, 0, r}$. As in the one-dimensional case, the Haar system
$g_{-1}, g_{0,0,1}, \dots, g_{0,0,2^d-1}, g_{1,0,1}, \dots$ (again,
indices are ordered lexicographically) is an unconditional Schauder
basis of all $L^p([0, 1]^d)$ spaces, for $1 < p < \infty$, see
\cite[Corollary 2.28 and Remark 2.29]{Trie}.

In dimension one, the Faber-Schauder functions $f_{-1}$, $f_{n,k}$ were
defined as the primitives of the functions $h_{-1}$, $h_{n,k}$. As the
multidimensional Haar functions $g_{-1}$, $g_{n,k,r}$ live in
$L^\infty([0, 1]^d)$, we can embed them in the space of strong charges
using the canonical map $L^d([0, 1]^d) \to \rmsch([0, 1]^d)$. This
gives rise to a sequence of strong charges, the first of them being
$\diff \calL$, followed by $\bomega_{0,0,1}, \dots,
\bomega_{0,0,2^d-1}, \bomega_{1, 0, 1}, \dots$ (indices are ordered
lexicographically), defined by
\[
\bomega_{n,k,r} = g_{n,k,r} \diff \calL, \qquad \text{for all relevant
  indices } n, k, r
\]
that we call the \emph{Faber-Schauder basis} of $\rmsch([0, 1]^d)$. It
was proven in \cite[Theorem 5.5]{BouaDePa} that it is indeed a
Schauder basis of $\rmsch([0, 1]^d)$. Here is a weakened version of
this result.

\begin{Theorem}
  \label{thm:faber}
  Each strong charge $\bomega \in \rmsch([0, 1]^d)$ admits the
  decomposition
  \[
  \bomega = \bomega([0, 1]^d) \diff \calL + \sum_{n=0}^\infty
  \sum_{k=0}^{2^{nd}-1} \sum_{r=1}^{2^d-1} a_{n,k,r}(\bomega)
  \bomega_{n,k,r}
  \]
  where the Faber-Schauder coefficients $a_{n,k,r}(\bomega)$ are given
  by the formula
  \[
  a_{n,k,r}(\bomega) = \calS^{-1}(\bomega)(g_{n,k,r}) = 2^{nd/2}
  \sum_{\ell=0}^{2^d-1} (A_d)_{r, \ell} \bomega(K_{n+1, 2^d k + \ell})
  \]
\end{Theorem}

\noindent It is often useful to rewrite the definition of the
coefficients $a_{n,k,r}(\bomega)$ in the form
\begin{equation}
  \label{eq:ankr_mf}
  \begin{pmatrix}
    2^{nd/2} \bomega(K_{n,k}) \\
    a_{n,k,1}(\bomega) \\
    \vdots \\
    a_{n,k,2^d-1}(\bomega)
  \end{pmatrix} = 2^{nd/2} A_d \begin{pmatrix}
    \bomega(K_{n+1, 2^d k}) \\
    \bomega(K_{n+1, 2^d k +1}) \\
    \vdots \\
    \bomega(K_{n+1, 2^d k + 2^d - 1})
  \end{pmatrix}.
\end{equation}
In the equality above, we used that the $0$-th row of $A_d$ is filled
with ones to compute the $0$-th coefficient.

\subsection{A locality property}
The one-dimensional Faber-Schauder functions have the pleasant
property of having localized supports. Owing to this, it is possible
to compute the norm of a linear combination of Faber-Schauder
functions belonging to the same generation, indeed
\[
\left\| \sum_{k=0}^{2^n-1} a_k f_{n,k} \right\|_\infty =
\frac{1}{2^{n/2+1}} \max_k \lvert a_{k}\rvert.
\]
A similar result, stated below, holds for the strong charges
$\bomega_{n,k,r}$ of generation $n$, that are localized in space as
well, each $\bomega_{n,k,r}$ being supported in the cube $K_{n,k}$.

Let us observe, by the way, that the alternative method of
antidifferentiating the functions $g_{n,k,r}$, through integration
over rectangles $(x_1, \dots, x_d) \mapsto \int_0^{x_1} \cdots
\int_0^{x_d} g_{n,k,r}$ leads to functions that do not exhibit these
good localization properties at all. The following proposition is
proven in \cite[Proposition 6.4]{BouaDePa}.

\begin{Proposition}
  \label{prop:maj1}
  There is a constant $C_{\ref{prop:maj1}} \geq 0$ such that, for all
  $n \geq 0$ and scalars $(a_{k,r})$, we have
  \[
  \left\| \sum_{k=0}^{2^{nd}-1} \sum_{r=1}^{2^d-1} a_{k,r}
  \bomega_{n,k,r} \right\| \leq C_{\ref{prop:maj1}} 2^{n(d/2-1)}
  \max_{k,r} \lvert a_{k,r}\rvert.
  \]
\end{Proposition}

\section{Hölder Charges on $[0, 1]^d$}
\label{sec:holder}

\subsection{Definition}
  \label{e:defHolder}
  Throughout this section, we fix a real number $\gamma$ such that
  \[
  \frac{d-1}{d} < \gamma \leq 1.
  \]
  A charge $\bomega \in \rmch([0, 1]^d)$ is \emph{$\gamma$-Hölder}
  whenever there is a constant $\kappa \geq 0$ such that $\lvert
  \bomega(K) \rvert \leq \kappa \lvert K \rvert^\gamma$ for all dyadic
  cubes $K \subset [0, 1]^d$. The linear space of $\gamma$-Hölder
  charges is denoted $\rmsch^\gamma([0, 1]^d)$ (the reason behind this
  notation is that Hölder charges are automatically strong, as will be
  seen in Corollary~\ref{cor:holder=>strong}) and we norm it with
  \[
  \| \bomega \|_\gamma = \sup \left\{ \frac{\lvert\bomega(K)\rvert}{\lvert
    K\rvert^\gamma} : K \text{ is a dyadic cube} \subset [0, 1]^d
  \right\}.
  \]
  In addition, we set $\|\bomega\|_\infty = \infty$ if $\bomega \in
  \rmch([0, 1]^d)$ is not $\gamma$-Hölder.
  
  In the definition, we have restricted the Hölder exponent $\gamma$
  from taking values less than or equal to $(d-1)/d$. This is because
  it appears that there is no meaningful theory of Hölder charges for
  such exponents. Furthermore, the definition we provide here is
  temporary. It has the drawback of being formulated in terms of
  dyadic cubes, which is not well-suited for
  generalization. In~\ref{e:holderIntrinseque}, we will provide a
  definition of Hölder charges on arbitrary $BV$-sets and invariant
  under bi-Lipschitz transformations, which will, of course, be
  equivalent to the current one over $[0, 1]^d$.

\subsection{Hölder charges in dimension $d=1$}
  \label{e:chargesdim1}
  In dimension $d=1$, Hölder exponents are allowed to range over
  $\gamma \in \mathopen{(}0, 1\mathclose{]}$. We claim that a charge
    $\bomega = \diff \varv$ is $\gamma$-Hölder if and only if the
    function $\varv \in C_0([0, 1])$ is $\gamma$-Hölder
    continuous. This could be obtained as a simple corollary of the
    subsequent Theorem~\ref{thm:holdintr}, though this claim can be
    proven by an easier chaining argument. In fact, it is well-known
    that a continuous function $\varv \colon [0, 1] \to \R$ is
    $\gamma$-Hölder continuous if and only if there is a constant
    $\kappa \geq 0$ such that
  \[
  \left| \varv\left( \frac{k+1}{2^n} \right) - \varv\left(
  \frac{k}{2^n} \right) \right| \leq \frac{\kappa}{2^{n\gamma}}
  \]
  for all integers $n \geq 0$ and $0 \leq k \leq 2^n-1$, see for
  example \cite[Lemma 2.10]{LeGa}.

\subsection{Strong charges with densities}
  \label{e:densityHolderC}
  An example worth mentioning is the one given by charges $g \diff
  \calL$ having a density $g \in L^{1/(1 - \gamma)}([0, 1]^d)$ (if
  $\gamma = 1$, then $1/(1-\gamma) = \infty$). For every dyadic cube
  $K$, the Hölder inequality gives
  \[
  \lvert (g \diff \calL)(K)\rvert \leq \left\lvert \int_K g
  \right\rvert \leq \lvert K\rvert^\gamma \|g\|_{1/(1 - \gamma)}
  \]
  so that $g \diff \calL \in \mathrm{sch}^\gamma([0, 1]^d)$ and $\|g
  \diff \calL\|_\gamma \leq \|g\|_{1/(1 - \gamma)}$.

\subsection{Construction of Hölder charges}
The upcoming lemma will be our main ingredient to build
$\gamma$-Hölder charges. It will be used extensively in the
sequel. The reader interested in a quick construction of the Young
integral, based on a variant of the sewing lemma, can jump straight to
Theorem~\ref{thm:sewing} after checking out the proof of this lemma.

\begin{Lemma}\label{prop:gammaHolder=>strong}
  Let $\bomega \colon \calC_{\mathrm{dyadic}}([0, 1]^d) \to \R$ be a
  map. Suppose that
  \begin{itemize}
  \item[(A)] $\bomega$ is additive: for each dyadic cube $K$,
    \[
    \bomega(K) = \sum_{L \text{ children of } K} \bomega(L).
    \]
  \item[(B)] there is a constant $\kappa \geq 0$ such that $\lvert
    \bomega(K)\rvert \leq \kappa \lvert K\rvert^\gamma$ for all dyadic cubes $K$.
  \end{itemize}
  Then $\bomega$ has a unique extension to $\niceBV([0, 1]^d)$ that is a
  strong charge, still denoted $\bomega$ (of course, $\bomega$ is
  $\gamma$-Hölder) and $\|\bomega\| \leq
  C_{\ref{prop:gammaHolder=>strong}} \|\bomega\|_{\gamma}$, where
  $C_{\ref{prop:gammaHolder=>strong}}$ depends solely on $d$ and
  $\gamma$.
\end{Lemma}

\begin{proof}
  First note that the extension of $\bomega$ to a charge is
  necessarily unique, by the De Giorgi approximation
  theorem~\ref{thm:DeGiorgi}. Consequently, we focus on proving its
  existence.
  
  Define the coefficients
  \begin{equation}
    \label{eq:defankr}
    \alpha_{n,k,r} = 2^{nd/2} \sum_{\ell=0}^{2^d -1} (A_d)_{r, \ell}
    \bomega(K_{n+1, 2^d k + \ell})
  \end{equation}
  for all relevant indices $n,k, r$. Using assumption (B), we derive
  that
  \[
  \lvert \alpha_{n,k,r}\rvert \leq 2^{nd/2} 2^d \kappa \left(
  \frac{1}{2^{(n+1)d}} \right)^{\gamma} \leq C \kappa 2^{nd \left(
    \frac{1}{2} - \gamma \right)}.
  \]
  Therefore, for all $n \geq 0$, the strong charge defined by
  \[
  \bomega_n = \sum_{k=0}^{2^{nd}-1} \sum_{r = 1}^{2^d - 1}
  \alpha_{n,k,r} \bomega_{n,k,r}
  \]
  has a norm bounded by
  \[
  \|\bomega_n\| \leq C_{\ref{prop:maj1}} 2^{n \left( \frac{d}{2}- 1
    \right)} \max_{k, r} \lvert \alpha_{n,k,r}\rvert \leq C\kappa
  2^{n(d - 1 - d\gamma)}.
  \]
  This allows to introduce
  \[
  \tilde{\bomega} = \bomega([0, 1]^d) \diff \calL + \sum_{n=0}^\infty
  \bomega_n.
  \]
  Indeed, as $\gamma > (d-1)/d$, the last series converges in $\rmsch([0,
    1]^d)$ and
  \begin{equation}
    \label{eq:schgamma->sch}
    \|\tilde{\bomega}\| \leq \lvert \bomega([0, 1]^d)\rvert \, \|\diff
    \calL\| + C\kappa \sum_{n=0}^\infty 2^{n(d - 1 - d\gamma)} \leq
    C\kappa.
  \end{equation}
  Observe that we have used the hypothesis that $\gamma >
  (d-1)/d$.

  Next we claim that the strong charge $\tilde{\bomega}$ agrees with
  $\bomega$ on dyadic cubes, and thus, by finite additivity, on dyadic
  figures as well. This is shown by induction on the generation
  numbers of cubes. The base case follows from
  $\tilde{\bomega}(K_{0,0}) =\tilde{\bomega}([0, 1]^d) = \bomega([0,
    1]^d) = \bomega(K_{0,0})$ (as $\bomega_{n,k,r}([0, 1]^d) = 0$ for
  all relevant indices $n,k,r$). For the induction step, suppose the
  result holds for cubes of generation number $n$ and note that, for
  all $k \in \{0, \dots, 2^{nd} - 1\}$,
  \begin{equation*}
    \begin{pmatrix}
      2^{nd/2} \ind_{K_{n,k}} \\ g_{n, k, 1} \\ \vdots \\ g_{n, k, 2^d
        -1}
    \end{pmatrix} =
    2^{nd/2} A_d \begin{pmatrix} \ind_{K_{n+1, 2^d k}} \\ \ind_{K_{n+1,
          2^d k + 1}} \\ \vdots \\ \ind_{K_{n+1, 2^d k + 2^d - 1}}
    \end{pmatrix}
  \end{equation*}
  (We used that the zeroth row of $A_d$ is filled with ones).
  Applying the strong charge functional $\calS^{-1}(\tilde{\bomega})$
  on both sides and the induction hypothesis, one obtains
  \begin{equation*}
    \begin{pmatrix}
      2^{nd/2} \bomega(K_{n,k}) \\ \alpha_{n, k, 1} \\ \vdots
      \\ \alpha_{n, k, 2^d -1}
    \end{pmatrix} =
    2^{nd/2} A_d \begin{pmatrix} \tilde{\bomega}(K_{n+1, 2^d k})
      \\ \tilde{\bomega}(K_{n+1, 2^d k + 1}) \\ \vdots
      \\ \tilde{\bomega}(K_{n+1, 2^d k + 2^d - 1})
    \end{pmatrix}.
  \end{equation*}
  On the other hand, using (A) and \eqref{eq:defankr},
  \begin{equation*}
    \begin{pmatrix}
      2^{nd/2} \bomega(K_{n,k}) \\ \alpha_{n, k, 1} \\ \vdots
      \\ \alpha_{n, k, 2^d -1}
    \end{pmatrix} =
    2^{nd/2} A_d \begin{pmatrix} \bomega(K_{n+1, 2^d k})
      \\ \bomega(K_{n+1, 2^d k + 1}) \\ \vdots
      \\ \bomega(K_{n+1, 2^d k + 2^d - 1})
    \end{pmatrix}
  \end{equation*}
  We conclude with the inversibility of $A_d$ and the arbitrariness of
  $k$ that $\bomega$ and $\tilde{\bomega}$ coincide on all dyadic
  cubes of generation $n+1$.

  Thus, $\tilde{\bomega}$ is the required extension of
  $\bomega$. Denoting this extension by $\bomega$ instead of
  $\tilde{\bomega}$, we note that, in the preceding computations,
  $\kappa$ could be set to $\|\bomega\|_\gamma$,
  and~\eqref{eq:schgamma->sch} gives the inequality $\|\bomega\| \leq
  C_{\ref{prop:gammaHolder=>strong}}\|\bomega\|_\gamma$, with
  \[
  C_{\ref{prop:gammaHolder=>strong}} = \| \diff \calL \| + C
  \sum_{n=0}^\infty 2^{n(d - d\gamma - 1)}
  \]
  meaning that the inclusion mapping $\rmsch^\gamma([0, 1]^d) \to
  \rmsch([0, 1]^d)$ is continuous.
\end{proof}

\begin{Corollary}
  \label{cor:holder=>strong}
  A $\gamma$-Hölder charge on $[0, 1]^d$ is strong.
\end{Corollary}

\begin{proof}
  Apply Lemma~\ref{prop:gammaHolder=>strong} to the restriction of a
  $\gamma$-Hölder charge to $\calC_{\mathrm{dyadic}}([0, 1]^d)$.
\end{proof}

\begin{Corollary}
  $\rmsch^\gamma([0, 1]^d)$ is a Banach space. 
\end{Corollary}

\begin{proof}
  A Cauchy sequence in $\rmsch^\gamma([0, 1]^d)$ converges pointwise
  on $\calC_{\mathrm{dyadic}}([0, 1]^d)$ to a map $\bomega \colon
  \calC_{\mathrm{dyadic}}([0, 1]^d) \to \R$ that readily satisfies the
  conditions (A) and (B) of
  Lemma~\ref{prop:gammaHolder=>strong}. Therefore, $\bomega$ extends
  to a $\gamma$-Hölder charge, that is easily seen to be the limit of
  the Cauchy sequence.
\end{proof}

\subsection{Decay of Faber-Schauder coefficients}

As was hinted in the proof of
Lemma~\ref{prop:gammaHolder=>strong}, the Faber-Schauder
coefficients of $\gamma$-Hölder charges exhibit a specific decay
rate. In fact, a strong charge $\bomega$ is $\gamma$-Hölder if and only if
$\max_{k,r} \lvert a_{n,k,r}(\bomega)\rvert = O\left(
2^{nd\left(\frac{1}{2} - \gamma \right)} \right)$. More precisely, we have

\begin{Lemma}
  \label{lemma:decay}
  Suppose $\gamma \neq 1$.  There is a constant $C_{\ref{lemma:decay}}
  > 0$ such that, for all $\bomega \in \rmsch^\gamma([0, 1]^d)$ and $n
  \geq 0$,
  \[
  \frac{\|\bomega\|_\gamma}{C_{\ref{lemma:decay}}} \leq \max \left\{
  \lvert \bomega([0, 1]^d) \rvert, \sup_{n,k,r} 2^{nd\left( \gamma -
    \frac{1}{2} \right)} \lvert a_{n,k,r}(\bomega)\rvert \right\} \leq
  C_{\ref{lemma:decay}} \|\bomega\|_\gamma.
  \]
\end{Lemma}

\begin{proof}
  The upper bound follows from $\lvert \bomega([0, 1]^d) \rvert \leq
  \|\bomega\|_\gamma$ and
  \begin{align*}
    \lvert a_{n,k,r}(\bomega) \rvert & = 2^{nd/2} \left\lvert
    \sum_{\ell=0}^{2^d-1} (A_d)_{r, \ell} \bomega(K_{n+1,2^dk + \ell})
    \right\rvert \\ & \leq 2^{nd/2} 2^d \|\bomega\|_\gamma \left(
    \frac{1}{2^{(n+1)d}}\right)^\gamma.
  \end{align*}
  As for the lower bound, let us call
  \[
  M = \sup_{n,k,r} 2^{nd\left(\gamma - \frac{1}{2} \right)} \lvert
  a_{n,k,r}(\bomega) \rvert
  \]
  and let us define the sequence $(w_n)$
  \[
  w_n = \max \left\{ 2^{nd\gamma} \lvert \bomega(K_{n,k}) \rvert : k= 0,
  \dots, 2^{nd}-1 \right\}, \qquad n\geq 0
  \]
  so that $\|\bomega\|_\gamma = \|w\|_\infty$. Recall that $(A_d)^2 =
  2^d I$, where $I$ is the identity matrix of order $2^d$. Multiplying
  by $2^{-d(1 - \gamma)} 2^{nd \left( \gamma - \frac{1}{2}
    \right)}A_d$ in~\eqref{eq:ankr_mf} yields
  \begin{align*}
  2^{(n+1)d\gamma} \begin{pmatrix}
    \bomega(K_{n+1, 2^d k}) \\
    \bomega(K_{n+1, 2^d k + 1}) \\
    \vdots \\
    \bomega(K_{n+1, 2^d k + 2^d - 1})
  \end{pmatrix} & =
  2^{-d(1 - \gamma)} 2^{nd \left( \gamma - \frac{1}{2}
    \right)}A_d\begin{pmatrix} 2^{nd/2} \bomega(K_{n,k}) \\ a_{n,k,1}(\bomega)
  \\ \vdots \\ a_{n,k,2^d-1}(\bomega)
  \end{pmatrix} \\
  & = \frac{1}{2^{d(1 - \gamma)}}\begin{pmatrix} 2^{nd\gamma}
    \bomega(K_{n,k}) \\ 2^{nd\left(\gamma - \frac{1}{2} \right)}
    a_{n,k,1}(\bomega) \\ \vdots \\ 2^{nd\left(\gamma - \frac{1}{2}
      \right)} a_{n,k,2^d - 1}(\bomega)
  \end{pmatrix}
  \end{align*}
  from which we deduce
  \[
  w_{n+1} \leq \frac{w_n}{2^{d(1 - \gamma)}} + (2^d-1)
  2^{-d(1-\gamma)}M
  \]
  By an easy proof by induction, it is now possible to prove that
  \[
  w_n \leq \max \left\{ w_0, \frac{(2^d-1)2^{-d(1-\gamma)}M}{1 -
    2^{-d(1 - \gamma)}} \right\}
  \]
  for all $n \geq 0$. As $w_0 = \lvert \bomega([0 ,1]^d) \rvert$, this
  concludes the proof of the lower bound.
\end{proof}

\subsection{Approximation by strong charges with $L^\infty$ densities}

In fact, Lemma~\ref{lemma:decay} tells us that the space
$\rmsch^\gamma([0, 1]^d)$ is secretly isomorphic to $\ell^\infty$, as
each $\gamma$-Hölder strong charge is determined by the sequence of
bounded coefficients formed by $\bomega([0, 1]^d)$ and
$2^{nd\left(\gamma - \frac{1}{2} \right)} a_{n,k,r}(\bomega)$, for all
relevant indices $n,k,r$. The situation is completely analogous to
what is known for the space of Hölder functions,
see~\cite{Cies}.

\begin{Corollary}
  \label{cor:approx}
  There is a constant $C_{\ref{cor:approx}}$ such that for all
  $\bomega \in \rmsch^\gamma([0, 1]^d)$, one can find a sequence
  $(g_p)$ in $L^\infty([0, 1]^d)$ satisfying
  \[
  \sup_{p} \|g_p \diff \calL\|_\gamma \leq C_{\ref{cor:approx}}
  \|\bomega\|_\gamma \text{ and } \lim_{p \to \infty} \|g_p \diff
  \calL - \bomega\| = 0.
  \]
\end{Corollary}

\begin{proof}
  It suffices to consider the sequence formed by the
  \[
  g_p = \bomega([0, 1]^d) + \sum_{n=0}^p \sum_{k=0}^{2^{nd}-1}
  \sum_{r=1}^{2^d - 1} a_{n,k,r}(\bomega) g_{n,k,r}. 
  \]
\end{proof}

\subsection{An Arzelà-Ascoli-type result}

Next, we prove an Arzelà-Ascoli-type result for
Hölder charges, that improves pointwise convergence to uniform
convergence for bounded sequences in $\rmsch^\gamma([0, 1]^d)$.

\begin{Proposition}
  \label{prop:weak}
  Let $(\bomega_p)$ be a sequence that is bounded in $\rmsch^\gamma([0,
    1]^d)$, and $\bomega \in \rmsch([0, 1]^d)$ such that $\bomega_p(K) \to
  \bomega(K)$ for all dyadic cubes $K$. Then $\bomega \in \rmsch^\gamma([0,
    1]^d)$ and $\|\bomega_p - \bomega\| \to 0$.
\end{Proposition}

\begin{proof}
  By Lemma~\ref{lemma:decay},
  \[
  \lvert a_{n,k,r}(\bomega_p) \rvert \leq C_{\ref{lemma:decay}}
  2^{nd\left(\frac{1}{2} - \gamma \right)} \sup_{p} \|\bomega_p\|_\gamma
  \]
  for all relevant indices $n,k,r$ and all $p$. By the definition of
  the $a_{n,k,r}$ functionals~\eqref{eq:ankr_mf}, we have
  $a_{n,k,r}(\bomega_p) \to a_{n,k,r}(\bomega)$ (and $\bomega_p([0
    ,1]^d) \to \bomega([0, 1]^d)$). Letting $p \to \infty$,
  \[
  \lvert a_{n,k,r}(\bomega) \rvert \leq C_{\ref{lemma:decay}}
  2^{nd\left(\frac{1}{2} - \gamma \right)} \sup_{p}
  \|\bomega_p\|_\gamma.
  \]
  This means that $\bomega \in \rmsch^\gamma([0, 1]^d)$, by
  Lemma~\ref{lemma:decay}. By Proposition~\ref{prop:maj1},
  \begin{equation}
    \label{eq:diffmupmu}
    \|\bomega_p - \bomega\| \leq \lvert (\bomega_p - \bomega)([0,
      1]^d) \rvert \|\diff \calL\| + C_{\ref{prop:maj1}}
    \sum_{n=0}^\infty 2^{n(d/2-1)} \max_{k,r} \lvert
    a_{n,k,r}(\bomega_p - \bomega) \rvert.
  \end{equation}
  We have the upper bound
  \[
  C_{\ref{prop:maj1}} 2^{n(d/2-1)} \max_{k,r} \lvert
  a_{n,k,r}(\bomega_p - \bomega) \rvert \leq \left( C \sup_p
  \|\bomega_p\|_\gamma \right) 2^{n(d - 1 - d\gamma)}
  \]
  where $\sum_{n=0}^\infty 2^{n(d-1-d\gamma)} < \infty$. As a result,
  we can let $p$ tend to $\infty$ term by term
  in~\eqref{eq:diffmupmu}, by the Lebesgue dominated convergence
  theorem, and conclude that $\bomega_p \to \bomega$ in $\rmsch([0,
    1]^d)$.
\end{proof}

\subsection{Isotropic behavior of Hölder charges}

Every charge is entirely determined by its values on dyadic
cubes. Therefore, it is natural to inquire whether the Hölder
quantitative control over dyadic cubes can be extended to arbitrary
$BV$-sets. An affirmative answer to this problem is provided by
Theorem~\ref{thm:holdintr}.

\begin{Theorem}
  \label{thm:holdintr}
  Let $\bomega \in \rmsch^\gamma([0, 1]^d)$. For all $B \in \niceBV([0,
    1]^d)$, we have
  \[
  \lvert\bomega(B)\rvert \leq \frac{C_{\ref{thm:holdintr}}
    \|\bomega\|_\gamma}{(\operatorname{isop} B)^{d(1 - \gamma)}} \lvert B
  \rvert^\gamma
  \]
  where $C_{\ref{thm:holdintr}}$ is a constant depending solely on $d$
  and $\gamma$.
\end{Theorem}

\begin{proof}
  Let $\bomega$ be a $\gamma$-Hölder charge. Let $B \in \niceBV([0, 1]^d)$
  a $BV$-set that is supposed to be non-negligible. First we treat the
  case where $\|B\| \leq |B|$. Set $\alpha = d \gamma - (d - 1)$ and
  note that $\alpha \in (0, 1)$. Then
  \[
  \lvert \bomega(B) \rvert \leq \|\bomega\| \|B\| \leq
  C_{\ref{prop:gammaHolder=>strong}} \|\bomega\|_\gamma \, \|B\|^{1 -
    \alpha} \, \lvert B \rvert^\alpha =
  C_{\ref{prop:gammaHolder=>strong}}\|\bomega\|_\gamma \left( \frac{
    \lvert B \rvert^{\frac{d-1}{d}}}{\operatorname{isop} B}\right)^{1
    - \alpha} \lvert B \rvert^\alpha.
  \]
  Therefore, we obtain
  \begin{equation}
    \label{eq:result1}
    \lvert \bomega(B) \rvert \leq C_{\ref{prop:gammaHolder=>strong}} \|\bomega\|_\gamma
  \left(\frac{1}{\operatorname{isop} B}\right)^{d(1 - \gamma)}
  |B|^\gamma
  \end{equation}
  for any $B$ such that $\|B\| \leq |B|$.

  We now turn to the case where $|B| \leq \|B\|$. Let us decompose the
  strong charge
  \[
  \bomega = \bomega([0, 1]^d) \diff \calL + \sum_{n=0}^\infty
  \bomega_n, \qquad \text{where } \bomega_n = \sum_{k=0}^{2^{nd}-1}
  \sum_{r=1}^{2^d - 1} a_{n,k,r}(\bomega) \bomega_{n,k,r}
  \]
  Recall that
  \[
  a_{n,k,r}(\bomega) = 2^{\frac{nd}{2}} \sum_{\ell=0}^{2^d - 1} (A_d)_{r, \ell}
  \bomega(K_{n+1, 2^d k + \ell})
  \]
  Therefore, since $\bomega$ is $\gamma$-Hölder, we have
  \begin{equation}
    \label{eq:ankr}
    \lvert a_{n,k,r}(\bomega) \rvert \leq \|\bomega\|_\gamma 2^{\frac{nd}{2}} 2^d
  \left(\frac{1}{2^{(n+1)d}}\right)^\gamma \leq C \|\bomega\|_\gamma 2^{nd
    \left(\frac{1}{2} - \gamma \right)}
  \end{equation}
  for all relevant indices $n, k, r$. Combining this with the upper
  bound in Proposition~\ref{prop:maj1}, we derive
  \begin{align}
  |\bomega_n(B)| & \leq \|\bomega_n\| \, \|B\| \notag \\ & \leq
  C_{\ref{prop:maj1}} 2^{n(d/2 - 1)} \max_{k,r} |a_{n,k,r}(\bomega)|
  \, \| B \| \notag \\ & \leq C 2^{n(d - 1 - d\gamma)}
  \|\bomega\|_\gamma \|B\| \label{eq:estimanB1}
  \end{align}
  On the other hand, we have
  \[
  \bomega_n(B) = \sum_{k=0}^{2^{nd} - 1} \sum_{r=1}^{2^d - 1}
  a_{n,k,r}(\bomega) \int_B g_{n,k,r} 
  \]
  Using the fact that $\lvert g_{n,k,r} \rvert \leq 2^{nd/2}
  \ind_{K_{n,k}}$ and~\eqref{eq:ankr},
  \begin{align}
  |\bomega_n(B)| & \leq \sum_{k=0}^{2^{nd}-1} \sum_{r=1}^{2^d - 1}
  |a_{n,k,r}(\bomega)| \, 2^{\frac{nd}{2}} | B \cap K_{n,k} | \notag \\ & \leq C \|\bomega\|_\gamma (2^d
  - 1) 2^{nd \left( \frac{1}{2} - \gamma \right)} 2^{\frac{nd}{2}}
  \sum_{k=0}^{2^{nd} - 1} |B \cap K_{n,k}| \notag \\ & \leq
  C 2^{nd(1 - \gamma)} \|\bomega\|_\gamma |B|. \label{eq:estimanB2}
  \end{align}
  Let $N \geq 0$ be an integer. By~\eqref{eq:estimanB1}
  and~\eqref{eq:estimanB2}, we obtain
  \begin{align*}
    |\bomega(B)| & \leq |\bomega([0, 1]^d)| \, |B| + \sum_{n=0}^\infty
    |\bomega_n(B)| \\ & \leq |\bomega([0, 1]^d)| \, |B| + C \|\bomega\|_\gamma \left[
    \left(\sum_{n=0}^{N} 2^{nd(1 - \gamma)}\right) |B| + \left(
    \sum_{n=N+1}^\infty 2^{n(d-1-d\gamma)} \right) \|B\| \right] \\ & \leq
    |\bomega([0, 1]^d)| \, |B| + C \|\bomega\|_\gamma \left[ 2^{Nd(1 - \gamma)} |B| +
      2^{N(d-1 - d\gamma)} \|B\| \right]
  \end{align*}
  Now, we want to choose $N$ so as to minimize the preceding upper
  bound. Note carefully that $d(1 - \gamma) > 0$ and $d-1-d\gamma <
  0$. We take $N$ such that $2^N \leq \|B\| / |B| \leq 2^{N+1}$ (this
  is possible because of the assumption $\|B\| \geq |B|$). Then
  \begin{align}
  |\bomega(B)| & \leq |\bomega([0, 1]^d)| \, |B| + C \|\bomega\|_\gamma \left[
    \left( \frac{\|B\|}{|B|}\right)^{d(1 - \gamma)} |B| + \left(
    \frac{2 |B|}{\|B\|} \right)^{-(d - 1 - d\gamma)} \|B\| \right]
  \notag \\ & \leq |\bomega([0, 1]^d)| \, |B| + C \|\bomega\|_\gamma \|B\|^{d(1
    - \gamma)} |B|^{1 - d(1 - \gamma)} \notag \\ & \leq |\bomega([0, 1]^d)|
  \, |B| + \frac{C \|\bomega\|_\gamma}{(\operatorname{isop} B)^{d(1 -
      \gamma)}} |B|^\gamma \label{eq:result2}
  \end{align}
  The first term above is easily bounded by
  \begin{equation}
    \label{eq:result3}
  |\bomega([0, 1]^d)| \, |B| \leq \|\bomega\|_\gamma \lvert B \rvert \leq
  \frac{C\|\bomega\|_\gamma}{(\operatorname{isop} B)^{d(1 - \gamma)}} |B|
  \leq \frac{C\|\bomega\|_\gamma}{(\operatorname{isop} B)^{d(1 - \gamma)}}
  |B|^\gamma
  \end{equation}
  The second inequality stems from the isoperimetric inequality, and
  the third follows from $|B| \leq 1$. The desired result follows
  from~\eqref{eq:result1}, \eqref{eq:result2} and~\eqref{eq:result3}.
\end{proof}

\subsection{Intrinsic definition of Hölder charges}
  \label{e:holderIntrinseque}
  We can turn the characterization of Hölder charge, as given by the
  previous theorem, into a definition. We will say that a charge
  $\bomega \in \operatorname{ch}(A)$ is $\gamma$-Hölder over an
  arbitrary $B$-set $A$, whenever there is a constant $C \geq 0$ such
  that
  \begin{equation}
    \label{eq:defHH}
    \lvert \bomega(B) \rvert \leq \frac{C}{(\operatorname{isop}
      B)^{d(1 - \gamma)}} \lvert B \rvert^\gamma = C \|B\|^{d(1 -
      \gamma)} \lvert B \rvert^{d\gamma - (d-1)}
  \end{equation}
  for all $B \in \niceBV(A)$. Note that the condition~\eqref{eq:defHH}
  implies readily implies that $\bomega$ is continuous with respect to
  $BV$-convergence.

\section{The Young Integral}
\label{sec:young}

\subsection{Integration over $[0, 1]^d$}
  For any function $f \colon A \to \R$, whose domain is a subset $A$
  of $\R^d$, we define
  \[
  \rmLip^\beta(f) = \sup \left\{ \frac{\lvert f(x) - f(y)\rvert}{\lvert
    x-y\rvert^\beta} : x, y \in A, x \neq y \right\}.
  \]
  The function $f$ is said to be \emph{$\beta$-Hölder continuous}
  whenever $\rmLip^\beta(f) < \infty$. The set of $\beta$-Hölder
  continuous functions on $A$ is denoted $C^\beta(A)$.

\begin{Theorem}
  \label{thm:young}
  Let $\beta, \gamma \in (0, 1)$ be such that $\beta + d \gamma >
  d$. There exists a unique bilinear map
  \[
  \int_{[0, 1]^d} \colon C^\beta([0, 1]^d) \times \rmsch^\gamma([0, 1]^d) \to \R
  \]
  that satisfies the following properties (A) and (B):
  \begin{enumerate}
  \item[(A)] for any $f \in C^\beta([0, 1]^d)$ and $g \in
    L^{1/(1-\gamma)}([0, 1]^d)$, 
    \[
    \int_{[0, 1]^d} f ( g \diff \calL) = \int_{[0, 1]^d} fg
    \]
  \item[(B)] if $(f_p)$ is a sequence in $C^\beta([0, 1]^d)$ that
    converges uniformly to $f$, and $(\bomega_p)$ is a sequence in
    $\rmsch^\gamma([0, 1]^d)$ that converges to $\bomega$ strongly in
    $\rmsch([0, 1]^d)$, if $\sup \rmLip^\beta(f_p) < \infty$ and $\sup
    \|\bomega_p\|_\gamma < \infty$, then
    \[
    \lim_{p \to \infty} \int_{[0, 1]^d} f_p \, \bomega_p =
    \int_{[0, 1]^d} f \, \bomega.
    \]
  \end{enumerate}
  Furthermore, the integral thus defined, called the \emph{Young
    integral}, satisfies
  \begin{equation}
    \label{eq:eqD}
    \left\lvert \int_{[0, 1]^d} f \, \bomega \right\rvert \leq
    C_{\ref{thm:young}} \left( \|f\|_\infty + \rmLip^\beta(f) \right)
    \|\bomega\|_\gamma.
  \end{equation}
  for all $f \in C^\beta([0, 1]^d)$ and $\bomega \in \rmsch^\gamma([0,
    1]^d)$.
\end{Theorem}

\begin{proof}
  We first show the uniqueness of the integral under the conditions
  (A) and (B).  Let $\bomega \in \rmsch^\gamma([0, 1]^d)$ be a Hölder
  charge. For all $p \geq 0$, let us call $\bomega_p$ the partial
  Faber-Schauder series
  \[
  \bomega_p = \bomega([0, 1]^d) \diff \calL + \sum_{n=0}^p
  \sum_{k=0}^{2^{nd}-1} \sum_{r=1}^{2^d-1} a_{n,k,r}(\bomega)
  \bomega_{n,k,r}.
  \]
  By linearity and (A), we must have
  \[
  \int_{[0, 1]^d} f \, \bomega_p = \bomega([0, 1]^d) \int_{[0,1]^d} f +
  \sum_{n=0}^p \sum_{k=0}^{2^{nd}-1} \sum_{r=1}^{2^d-1} a_{n,k,r}(\bomega)
  \int_{[0, 1]^d} f g_{n,k,r}.
  \] 
  By Lemma~\ref{lemma:decay}, the sequence $(\bomega_p)$ is bounded in
  $\rmsch^\gamma([0, 1]^d)$. Of course $\bomega_p \to \bomega$ in $\rmsch([0,
    1]^d)$ (by Theorem~\ref{thm:faber}), and therefore we must have
  \begin{equation}
    \label{eq:yIf}
  \int_{[0, 1]^d} f \, \bomega = \lim_{p \to \infty} \int_{[0,
      1]^d} f \, \bomega_p
  \end{equation}
  or, equivalently,
  \begin{equation}
    \label{eq:youngIntegral}
    \int_{[0, 1]^d} f \, \bomega = \bomega([0, 1]^d) \int_{[0, 1]^d} f +
    \sum_{n=0}^\infty \sum_{k=0}^{2^{nd}-1} \sum_{r=1}^{2^d-1}
    a_{n,k,r}(\bomega) \int_{[0, 1]^d} fg_{n,k,r}.
  \end{equation}
  but we do not know yet whether the limit in the right-hand side
  of~\eqref{eq:yIf} exists, that is to say, the series
  in~\eqref{eq:youngIntegral} converges. To this end, we need to
  estimate the integrals
  \[
  \int_{[0, 1]^d} f g_{n,k,r} = 2^{nd/2} \sum_{\ell = 0}^{2^{d}-1} (A_d)_{r, \ell} \int_{K_{n+1, 2^d k + \ell}} f.
  \]
  Half of the coefficients in the $r$-th row of $A_d$ are equal to
  $1$, and the other half to $-1$. Therefore, we consider any partition
  \[
  \{0, 1, \dots, 2^d-1\} = \{\ell_0, \ell_1, \dots, \ell_{2^{d-1}-1}
  \} \cup \{\ell_0', \ell_1', \dots, \ell_{2^{d-1}-1}' \}
  \]
  where $(A_d)_{r, \ell_i} = 1$ and $(A_d)_{r, \ell_i'} = -1$ for all
  $0 \leq i \leq 2^{d-1}-1$. For each $i$, let $\varv_i$ be the vector
  such that $K_{n+1, 2^d k + \ell_i'} = K_{n+1, 2^d k + \ell_i} +
  \varv_i$. Then
  \begin{align}
  \left\lvert \int_{[0, 1]^d} f g_{n,k,r} \right\rvert & = 2^{nd/2}
  \left \lvert \sum_{i = 0}^{2^{d-1}-1} \left( \int_{K_{n+1, 2^d k +
      \ell_i}} f - \int_{K_{n+1, 2^d k + \ell_i'}} f \right)
  \right\rvert \notag \\ & \leq 2^{nd/2} \sum_{i=0}^{2^{d-1}-1} \int_{K_{n+1,
      2^d k + \ell_i}} \left\lvert f(x) - f(x + \varv_i) \right\rvert
  \diff x \notag \\ & \leq 2^{nd/2} 2^{d-1} 2^{-(n+1)d} \rmLip^\beta(f)
  \left( \operatorname{diam} K_{n,k}\right)^\beta \notag \\ & \leq C
  \rmLip^\beta(f) 2^{-n \left( \frac{d}{2} + \beta \right)} \label{eq:smallBeta}
  \end{align}
  We proceed with
  \begin{align*}
  \left\lvert \sum_{k=0}^{2^{nd}-1} \sum_{r=1}^{2^d-1} a_{n,k,r}(\bomega)
  \int_{[0, 1]^d} fg_{n,k,r} \right\rvert & \leq \sum_{k=0}^{2^{nd}-1}
  \left( C_{\ref{lemma:decay}} \|\bomega\|_\gamma 2^{nd \left( \frac{1}{2}
    - \gamma \right)}\right) \left( C \rmLip^\beta(f) 2^{-n
    \left(\frac{d}{2} + \beta\right)}\right) \\ & \leq C
  \|\bomega\|_\gamma \rmLip^\beta(f) 2^{n(d-d\gamma-\beta)}.
  \end{align*}
  As $\beta + d \gamma > d$, the series in~\eqref{eq:youngIntegral}
  converges, as desired:
  \begin{equation}
    \label{eq:youngeq}
    \sum_{n=0}^\infty \left\lvert \sum_{k=0}^{2^{nd}-1}
    \sum_{r=1}^{2^d-1} a_{n,k,r}(\bomega) \int_{[0, 1]^d} fg_{n,k,r}
    \right\rvert \leq C \rmLip^\beta(f) \|\bomega\|_\infty < \infty
  \end{equation}
  Thus, we take~\eqref{eq:youngIntegral} as our
  definition of the Young integral.

  Next we prove (A). Let $g \in L^{1/(1 - \gamma)}([0, 1]^d)$ be a
  density. Recall that the Haar system is a Schauder basis in
  $L^{1/(1-\gamma)}([0, 1]^d)$, so we can decompose $g$ as
  \begin{align*}
  g & = \left( \int_{[0, 1]^d} gg_{-1} \right) g_{-1} +
  \sum_{n=0}^\infty \sum_{k=0}^{2^{nd}-1} \sum_{r=1}^{2^d-1}
  \left(\int_{[0, 1]^d} gg_{n,k,r} \right) g_{n,k,r} \\ & = (g
  \diff \calL)([0, 1]^d) g_{-1} + \sum_{n=0}^\infty \sum_{k=0}^{2^{nd}-1}
  \sum_{r=1}^{2^d-1} a_{n,k,r}(g \diff \calL) g_{n,k,r}
  \end{align*}
  with convergence in $L^{1/(1-\gamma)}([0, 1]^d)$. Being continuous,
  the function $f$ is obviously regular enough for the functional $h
  \mapsto \int_{[0,1]^d} fh$ to be well-defined and continuous on
  $L^{1/(1-\gamma)}([0, 1]^d)$. Applying it to the preceding equality,
  we obtain
  \begin{align*}
  \int_{[0, 1]^d} fg & = (g \diff \calL)([0, 1]^d) \int_{[0, 1]^d} f +
  \sum_{n=0}^\infty \sum_{k=0}^{2^{nd}-1} \sum_{r=1}^{2^d-1}
  a_{n,k,r}(g \diff \calL) \int_{[0, 1]^d} fg_{n,k,r} \\ & = \int_{[0,
      1]^d} f (g \diff \calL).
  \end{align*}
  Next we prove (B). Consider a bounded sequence $(\bomega_p)$ in
  $\rmsch^\gamma([0, 1]^d)$ that converges to some $\bomega$ strongly in
  $\rmsch([0, 1]^d)$, as $p \to \infty$. And let $(f_p)$ be a sequence
  in $C^\beta([0, 1]^d)$ that converges uniformly to $f$, with $\sup
  \rmLip^\beta(f_p) < \infty$. It is clear that $\bomega \in \rmsch^\gamma([0,
    1]^d)$ and $f \in C^\beta([0, 1]^d)$.
  First note that, for any relevant indices $n,k,r$, we have
  \[
  \lim_{p \to \infty} a_{n,k,r}(\bomega_p) \int_{[0, 1]^d} f_p g_{n,k,r} =
  a_{n,k,r}(\bomega) \int_{[0, 1]^d} f g_{n,k,r}.
  \]
  This is for the $a_{n,k,r}$ are continuous on $\rmsch([0, 1]^d)$ and
  the convergence $f_p g_{n,k,r} \to f g_{n,k,r}$ is
  uniform. Similarly, one has
  \[
  \lim_{p \to \infty} \bomega_p([0, 1]^d) \int_{[0, 1]^d} f_p = \bomega([0, 1]^d) \int_{[0, 1]^d} f.
  \]
  Additionally, by Lemma~\ref{lemma:decay} and the
  inequality~\eqref{eq:smallBeta}, applied to $f_p$ instead of $f$, we
  have
  \begin{align*}
  \left \lvert a_{n,k,r}(\bomega_p) \int_{[0, 1]^d} f_pg_{n,k,r}\right\rvert
  & \leq C_{\ref{lemma:decay}} 2^{nd \left( \frac{1}{2} - \gamma
    \right)} \|\bomega_p\|_\gamma \left(C
  \|f_p\|_\beta 2^{-n\left( \frac{d}{2} + \beta \right)} \right) \\ &
  \leq C  \left( \sup_{p\geq 0} \rmLip^\beta(f_p) \|\bomega_p\|_\gamma \right)
  2^{-n(\beta + d\gamma)}.
  \end{align*}
  The right-hand side is summable, when summed over all relevant
  indices $(n,k,r)$. By the Lebesgue dominated convergence theorem, we
  conclude that
  \[
  \lim_{p \to \infty} \sum_{n=0}^\infty \sum_{k=0}^{2^{nd}-1}
  \sum_{r=1}^{2^d-1} a_{n,k,r}(\bomega_p) \int_{[0, 1]^d} f_pg_{n,k,r}
  = \sum_{n=0}^\infty \sum_{k=0}^{2^{nd}-1}
  \sum_{r=1}^{2^d-1} a_{n,k,r}(\bomega) \int_{[0, 1]^d} fg_{n,k,r}
  \]
  Now it is clear that $\int_{[0, 1]^d} f_p \, \bomega_p \to
  \int_{[0, 1]^d} f \, \bomega$ as $p \to \infty$.

  It remains to prove~\eqref{eq:eqD}. Clearly,
  \[
  \left\lvert \bomega([0, 1]^d) \int_{[0, 1]^d} f \right\rvert \leq \lvert
  \bomega([0, 1]^d) \rvert \|f\|_\infty \leq \|\bomega\|_\gamma \|f\|_\infty
  \]
  and the result follows by combining this inequality,
  \eqref{eq:youngIntegral} and~\eqref{eq:youngeq}.
\end{proof}

The formula~\eqref{eq:youngIntegral}, that may serve for the
definition of the Young integral, is of independent interest. It is
reminiscent of the non-causal integral defined
in~\cite[V.A]{CiesKerkRoyn}.

\subsection{Indefinite Young integration}

We now dedicate ourselves to improving the Young integral, this time
allowing us to integrate a Hölder continuous function with respect to
a Hölder charge on any $BV$-subset of $[0, 1]^d$. In other words, we
are seeking to define an indefinite Young integral $f \cdot \bomega$
(Theorem~\ref{thm:youngIndefinite}) of $f \in C^\beta([0, 1]^d)$ with
respect to $\bomega \in \rmsch^\gamma([0, 1]^d)$, which is a
charge. We are giving ourselves the freedom to use the notation
\[
\int_B f \, \bomega \text{ for } (f \cdot \bomega)(B).
\]
We will need the following lemma.

\begin{Lemma}
  \label{lemma:Phi}
  Let $\bomega \in \rmsch^\gamma([0, 1]^d)$ and $K \subset [0, 1]^d$
  be a cube, not necessarily dyadic. Let $\Phi \colon [0, 1]^d \to K$
  be the obvious map that sends $[0, 1]^d$ onto $K$, by translation
  and homothety.  Then $\bomega \circ \Phi \in \rmsch^\gamma([0,
    1]^d)$ and
  \[
  \| \bomega \circ \Phi \|_\gamma \leq C_{\ref{lemma:Phi}}
  \|\bomega\|_\gamma |K|^\gamma.
  \]
\end{Lemma}

\begin{proof}
  Let $L \subset [0, 1]^d$ be a dyadic cube. By
  Theorem~\ref{thm:holdintr}, it holds that
  \[
  \left\lvert (\bomega \circ \Phi)(L) \right\rvert =
  \left\lvert \bomega(\Phi(L)) \right\rvert \leq
  \frac{C_{\ref{thm:holdintr}}}{(\operatorname{isop} \Phi(L))^{d(1 -
      \gamma)}} \lvert \Phi(L) \rvert^\gamma
  \]
  Moreover, the isoperimetric coefficient of any cube, such as
  $\Phi(L)$, is just the constant $1/(2d)$. On top of that $|\Phi(L)|
  = |K| \, |L|$, so that
  \[
  \left\lvert (\bomega \circ \Phi)(L) \right\rvert \leq C
  \|\bomega\|_\gamma |K|^\gamma \, |L|^\gamma.
  \]
  This proves the claim.
\end{proof}

\begin{Theorem}[Indefinite Young integral]
  \label{thm:youngIndefinite}
  Let $\beta, \gamma \in (0, 1)$ be such that $\beta + d \gamma >
  d$. There exists a unique bilinear map
  \[
  C^\beta([0, 1]^d) \times \rmsch^\gamma([0, 1]^d) \to
  \rmsch^\gamma([0, 1]^d) \colon (f, \bomega) \mapsto f \cdot \bomega
  \]
  that satisfies the following properties (A) and (B):
  \begin{enumerate}
  \item[(A)] for any $f \in C^\beta([0, 1]^d)$ and $g \in
    L^{1/(1-\gamma)}([0, 1]^d)$, 
    \[
    f \cdot (g \diff \calL) = fg \diff \calL
    \]
  \item[(B)] if $(f_p)$ is a sequence in $C^\beta([0, 1]^d)$ that
    converges uniformly to $f$, and $(\bomega_p)$ is a sequence in
    $\rmsch^\gamma([0, 1]^d)$ that converges to $\bomega$ strongly in
    $\rmsch([0, 1]^d)$, if $\sup \rmLip^\beta(f_p) < \infty$ and $\sup
    \|\bomega_p\|_\gamma < \infty$, then
    \[
    \lim_{p \to \infty} f_p \cdot \bomega_p = f \cdot \bomega \text{ in
    } \rmsch([0, 1]^d).
    \]
  \end{enumerate}
  Furthermore, we have, for any dyadic cube $K$,
  \begin{equation}
    \label{eq:indefboundK}
    \lvert (f \cdot \bomega)(K) \rvert \leq
    C_{\ref{thm:youngIndefinite}} \left( \|f_{\mid K}\|_\infty +
    \lvert K \rvert^{\beta/d} \rmLip^\beta(f) \right) \| \bomega
    \|_\gamma \lvert K \rvert^\gamma.
  \end{equation}
  In particular,
  \begin{equation}
    \label{eq:indefbound}
    \left\| f \cdot \bomega \right\|_\gamma \leq
    C_{\ref{thm:youngIndefinite}} \left( \|f\|_\infty + \rmLip^\beta(f)
    \right) \| \bomega \|_\gamma.
  \end{equation}

\end{Theorem}

\begin{proof}
  The uniqueness is proven as in Theorem~\ref{thm:young}, where we
  established that any $\bomega \in \rmsch^\gamma([0, 1]^d)$ is the
  limit, in the sense of (B), of a sequence of charges $(g_p \diff
  \calL)$ with density with densities $g_p \in L^\infty([0,
    1]^d)$. Consequently, (A) and (B) together imply that $f \cdot
  \bomega$ must be the limit $\lim f_p g_p \diff \calL$ (whose
  existence is not yet established).

  Let $K \subset [0, 1]^d$ be a dyadic cube and $\bomega \in
  \rmsch^\gamma([0, 1]^d)$. We first prove that $\bomega \hel K$ is
  $\gamma$-Hölder. To this end, let $L \subset [0, 1]^d$ be another
  dyadic cube. Three situations can occur
  \begin{itemize}
  \item if $K \subset L$, then
    \[
    \frac{|(\bomega \hel K)(L)|}{|L|^\gamma} \leq
    \frac{|\bomega(K)|}{|K|^\gamma} \leq \|\bomega\|_\gamma
    \]
  \item if $K$ and $L$ are almost disjoint, then \[\frac{|(\bomega \hel K)(L)|}{|L|^\gamma} = 0 \]
  \item if $L \subset K$ then \[ \frac{|(\bomega \hel K)(L)|}{|L|^\gamma} =
    \frac{|\bomega(L)|}{|L|^\gamma} \leq \|\bomega\|_\gamma \]
  \end{itemize}
  This shows that $\bomega \hel K \in \rmsch^\gamma([0, 1]^d)$ and
  \begin{equation}
    \label{eq:helKLip}
    \|\bomega \hel K\|_\gamma \leq \|\bomega\|_\gamma.
  \end{equation}
  
  Now, for each dyadic cube $K$, we define
  \begin{equation}
    \label{eq:youngOnL}
    (f \cdot \bomega)(K) = \int_K f \, \bomega = \int_{[0, 1]^d} f \,
    \left( \bomega \hel K \right).
  \end{equation}
  It is clear that the set function $f \cdot \bomega$, at the moment
  defined on dyadic cubes, extends to a finitely additive function on
  $\calF_{\mathrm{dyadic}}([0, 1]^d)$. Our goal is to apply
  Lemma~\ref{prop:gammaHolder=>strong} and for that, we only need to
  check that $f \, \bomega$ satisfies
  \ref{prop:gammaHolder=>strong}(B).
  
  We let $\Phi \colon [0, 1]^d \to K$ be the obvious map that sends
  $[0, 1]^d$ onto $K$, by translation and homothety. By
  Lemma~\ref{lemma:Phi}, $\bomega \circ \Phi \in
  \rmsch^\gamma([0, 1]^d)$.  We assert that
  \begin{equation}
    \label{eq:youngKeq}
    \int_{K} f \, \bomega = \int_{[0, 1]^d} f \, \left( \bomega \hel
    K\right) = \int_{[0, 1]^d} f \circ \Phi \, (\bomega \circ
    \Phi)
  \end{equation}
  To prove~\eqref{eq:youngKeq}, first suppose $\bomega$ has a density
  $g \in L^{1/(1-\gamma)}([0, 1]^d)$, that is $\bomega = g \diff \calL$. Then
  \[
  \int_{[0, 1]^d} f \, \left( g \diff \calL \hel K\right) =
  \int_{[0, 1]^d} f \, \left( g \ind_K \diff \calL \right) = \int_K fg
  \]
  by Theorem~\ref{thm:young}(A). On the other hand, $(g \diff
  \calL) \circ \Phi = |K| (g \circ \Phi \diff \calL)$, as for
  any $B \in \niceBV([0, 1]^d)$, we have
  \[
  (g \diff \calL) \circ \Phi(B) =
  \int_{\Phi(B)} g = \lvert K \rvert \int_{B} g \circ \Phi = \lvert K
  \rvert (g \circ \Phi \diff \calL)(B)
  \]
  so that the right-hand side of~\eqref{eq:youngKeq} is simply
  \begin{equation}
    \label{eq:proofA}
    \int_{[0, 1]^d} f \circ \Phi \, \left( (g \diff \calL)
    \circ \Phi \right) = |K| \int_{[0, 1]^d} (f \circ \Phi) (g \circ
    \Phi) = \int_K fg.
  \end{equation}
  Therefore, \eqref{eq:youngKeq} is valid for charges with density. To
  prove it in the general case, let $\bomega \in \rmsch^\gamma([0,
    1]^d)$ and $(\bomega_p)$ a sequence of charges with a density in
  $L^{1/(1-\gamma)}([0, 1]^d)$, bounded in $\rmsch^\gamma([0, 1]^d)$,
  that converges to $\bomega$, strongly in $\rmsch([0,
    1]^d)$. By~\eqref{eq:helKLip} and Lemma~\ref{lemma:Phi} applied to
  the charges $\bomega_p$, the sequences $(\bomega_p \hel K)$ and $(
  \bomega_{p} \circ \Phi)$ are bounded in $\rmsch^\gamma([0,
    1]^d)$. Proposition~\ref{prop:weak} can be easily applied to show
  that $\bomega_p \hel K \to \bomega \hel K$ and $\bomega_{p}
  \circ \Phi \to \bomega \circ \Phi$ in $\rmsch([0,
    1]^d)$. By Theorem~\ref{thm:young}(B), we deduce
  that~\eqref{eq:youngKeq} is true for any $\bomega \in
  \rmsch^\gamma([0, 1]^d)$. From there, we use
  Theorem~\ref{thm:young}(D) to deduce that
  \[
  \left\lvert \int_K f \, \bomega \right\rvert \leq
  C_{\ref{thm:young}} \left( \|f \circ \Phi\|_\infty + \rmLip^\beta(f
  \circ \Phi) \right) \| \bomega \circ \Phi \|_\gamma
  \]
  As $\Phi$ is $\lvert K\rvert^{1/d}$-Lipschitz, $\rmLip^\beta(f \circ
  \Phi) \leq \lvert K\rvert^{1/d} \rmLip^\beta(f)$. Then by
  Lemma~\ref{lemma:Phi}, we conclude that~\eqref{eq:indefboundK}
  holds.  Therefore, Lemma~\ref{prop:gammaHolder=>strong} implies that
  $f \, \bomega$ extends to a $\gamma$-Hölder charge, as wished, and
  we have~\eqref{eq:indefbound}.

  It remains to prove (A) and (B). As was noticed
  in~\eqref{eq:proofA}, the charges $f \, (g \diff \calL)$ and $fg
  \diff \calL$ coincide on all dyadic cubes, if $g \in
  L^{1/(1-\gamma)}([0, 1]^d)$, so they are equal, which establishes
  (A).

  Finally, take sequences $(f_p)$ and $(\bomega_p)$ converging
  respectively to $f$ and $\bomega$, as in
  (B). By~\eqref{eq:indefbound}, applied to $(f_p, \bomega_p)$, the
  sequence $(f_p \cdot \bomega_p)$ is bounded in $\rmsch^\gamma([0,
    1]^d)$. For any dyadic cube $K$, the sequence of charges
  $(\bomega_p \hel K)$ is bounded in $\rmsch^\gamma([0, 1]^d)$ and
  converges to $\bomega\hel K$ in $\rmsch([0 , 1]^d)$, here again
  by~\eqref{eq:helKLip} and Proposition~\ref{prop:weak}. By
  Theorem~\ref{thm:young}(B),
  \[
  \lim_{p \to \infty} \int_{[0, 1]^d} f_p \, (\bomega_p \hel K) =
  \int_{[0, 1]^d} f \, (\bomega \hel K).
  \]
  This means that $\int_K f_p \, \bomega_p \to \int_K f \,
  \bomega$. As this happens for every dyadic cube $K$, (B) is now a
  consequence of Proposition~\ref{prop:weak}.
\end{proof}

\subsection{Young-Loeve-type estimate}

We proceed with a Young-Loeve-type estimate. Before that, we briefly
show that the indefinite integral $1 \cdot \bomega$ is just $\bomega$,
for $\bomega \in \rmsch^\gamma([0, 1]^d)$. To prove this, we could
argue by density, establishing the result first for charges with
densities in $L^{1/(1-\gamma)}$. Note however that $1 \cdot \bomega$
and $\bomega$ coincide on $[0, 1]^d$ by~\eqref{eq:youngIntegral}, and
more generally on any dyadic cube, this time
using~\eqref{eq:youngOnL}.

\begin{Theorem}
  \label{thm:youngLoeve}
  Let $f \in C^\beta([0, 1]^d)$ and $\bomega \in \rmsch^\gamma([0, 1]^d)$,
  with $\beta + d \gamma > d$. For any $B \in \niceBV([0, 1]^d)$ and
  $x \in B$, we have
  \begin{equation}
    \label{eq:loeve}
    \left\lvert \int_B f \, \bomega - f(x)\bomega(B) \right\rvert \leq
    C_{\ref{thm:youngLoeve}} \frac{\rmLip^\beta(f) \|\bomega\|_\gamma
      |B|^\gamma (\operatorname{diam} B)^\beta}{(\operatorname{isop}
      B)^{d(1 - \gamma)}}.
  \end{equation}
\end{Theorem}
First note that the condition $x \in B$ is not restrictive, as one
could replace $B$ by $B \cup \{x\}$ (of course, this modifies the
diameter of $B$). Also, the upper bound in~\eqref{eq:loeve} takes into
account the geometry of the set $B$ in a rather complicated
manner. For a cube $K \subset [0, 1]^d$, \eqref{eq:loeve} simplifies
to
\[
\left\lvert \int_K f \, \bomega - f(x)\bomega(K) \right\rvert \leq
C \lvert K \rvert^{\frac{\beta + d\gamma}{d}}
\]
and thus we see that the error between $\int_K f \, \bomega$ and $f(x)
\bomega(K)$ has an order greater than $1$.

\begin{proof}
  Without loss of generality, we can suppose that $f(x) = 0$.  There
  is a cube $K \subset [0, 1]^d$, not necessarily dyadic, of side
  length $\operatorname{diam} B$, that entirely contains $B$.  We let
  $\Phi \colon [0, 1]^d \to K$ be as usual the unique map that sends
  $[0, 1]^d$ onto $K$ by translation and homothety. We claim that
  \begin{equation}
    \label{eq:noImag}
    (f \cdot \bomega) \circ \Phi = (f \circ \Phi) \cdot
    (\bomega \circ \Phi).
  \end{equation}
  Note that each charge involved in this equality is $\gamma$-Hölder,
  as per Lemma~\ref{lemma:Phi}. We need only show that the charges
  in~\eqref{eq:noImag} coincide on the dyadic cubes $L$. In this
  regard, let $\tilde{\Phi} \colon [0, 1]^d \to L$ be the obvious
  affine map transporting $[0, 1]^d$ onto $L$. Then $\Phi \circ
  \tilde{\Phi}$ is the affine map transporting $[0, 1]^d$ onto the
  cube $\Phi(L)$. By~\eqref{eq:youngKeq}, applied twice,
  \begin{align*}
  \left( (f \cdot \bomega) \circ \Phi \right)(L) & = \int_{\Phi(L)} f
  \, \bomega \\ & = \int_{[0, 1]^d} (f \circ \Phi \circ \tilde{\Phi})
  \, ( \bomega \circ \Phi \circ \tilde{\Phi}) \\ & = \int_L (f \circ
  \Phi) \, (\bomega \circ \Phi)
  \end{align*}
  This proves~\eqref{eq:noImag}.

  Evaluating~\eqref{eq:noImag} at $\Phi^{-1}(B)$, one gets
  \[
  \int_B f \, \bomega = \int_{\Phi^{-1}(B)} (f \circ \Phi) \,
  (\bomega \circ \Phi)
  \]
  which implies, by Theorem~\ref{thm:holdintr}, that
  \[
  \left\lvert \int_B f \, \bomega \right\rvert \leq
  \frac{C_{\ref{thm:holdintr}}}{(\operatorname{isop}
    \Phi^{-1}(B))^{d(1 - \gamma)}} \left\| (f \circ \Phi) \cdot
  (\bomega \circ \Phi) \right\|_\gamma \lvert
  \Phi^{-1}(B)\rvert^\gamma.
  \]
  Observe that $B$ and $\Phi^{-1}(B)$ share the same isoperimetric
  coefficient. Also, the map $\Phi^{-1}$ multiplies volumes by $1/ \lvert K \rvert$. By~\eqref{eq:indefbound} and Lemma~\ref{lemma:Phi}, we deduce
  \begin{equation}
    \label{eq:almostLoeve}
  \left\lvert \int_B f \, \bomega \right\rvert \leq
  \frac{C}{(\operatorname{isop} B)^{d(1 - \gamma)}} \left( \|f \circ
  \Phi\|_\infty + \rmLip^\beta(f \circ \Phi) \right) \|\bomega\|_\gamma
  \lvert B \rvert^\gamma.
  \end{equation}
  However, $\Phi$ is $(\operatorname{diam} B)$-Lipschitz, thus
  $\rmLip^\beta(f \circ \Phi) \leq \rmLip^\beta(f) (\operatorname{diam}
  B)^{\beta}$. In addition, since $f \circ \Phi$ vanishes at
  $\Phi^{-1}(x)$, one has
  \[
  \|f \circ \Phi \|_{\infty} \leq \rmLip^\beta (f \circ \Phi)
  (\operatorname{diam} [0, 1]^d)^\beta \leq C \rmLip^\beta(f)
  (\operatorname{diam} B)^\beta.
  \]
  Consequently, \eqref{eq:loeve} follows from~\eqref{eq:almostLoeve}.
\end{proof}

\subsection{Construction by means of a sewing lemma}

We conclude this section with an alternative construction of the Young
indefinite integral, based on a variation of the sewing lemma.

\begin{Theorem}[Charge sewing lemma]
  \label{thm:sewing}
  Let $\boldsymbol{\eta} \colon \calC_{\mathrm{dyadic}}([0, 1]^d) \to
  \R$ be a function. Suppose there are constants $\kappa \geq 0$ and
  $\varepsilon > 0$ such that:
  \begin{itemize}
  \item[(A)] almost finite additivity: for each dyadic cube $K$,
    \[
    \left\lvert \boldsymbol{\eta}(K) - \sum_{L \text{ children of } K} \boldsymbol{\eta}(L)
    \right\rvert \leq \kappa \lvert K \rvert^{1 + \varepsilon}
    \]
  \item[(B)] for each dyadic cube, $|\boldsymbol{\eta}(K)| \leq \kappa \lvert K
    \rvert^\gamma$.
  \end{itemize}
  There is a unique charge $\btheta \in \rmsch^\gamma([0, 1]^d)$ such
  that $\lvert \boldsymbol{\eta}(K) - \btheta(K)\rvert \leq
  C_{\ref{thm:sewing}} \kappa \lvert K \rvert^{1+ \varepsilon}$, for
  some constant $C_{\ref{thm:sewing}} \geq 0$.
\end{Theorem}

\begin{proof}
  We first prove uniqueness. For any dyadic cube $K$, we let
  $\calC_p(K)$ be the set of its $2^{pd}$ grandchildren that have
  diameter $2^{-p} \operatorname{diam} K$. If $\btheta, \tilde{\btheta}
  \in \rmsch^\gamma([0, 1]^d)$ are two solutions, then for any dyadic
  cube $K$ and $p \geq 0$, we estimate
  \[
  \lvert \btheta(K) - \tilde{\btheta}(K) \rvert \leq \sum_{L \in
    \calC_p(K)} \lvert \btheta(L) - \tilde{\btheta}(L) \rvert \leq
  2C_{\ref{thm:sewing}}\kappa 2^{pd} \left( \frac{\lvert K
    \rvert}{2^{pd}} \right)^{1+\varepsilon}.
  \]
  By letting $p \to \infty$, we deduce that $\btheta(K) =
  \tilde{\btheta}(K)$. As $K$ is any dyadic cube, this implies that
  $\btheta = \tilde{\btheta}$.

  Next, for any integer $p \geq 0$, we define $\btheta_p \colon
  \calC_{\mathrm{dyadic}}([0, 1]^d) \to \R$ by
  \[
  \btheta_p(K) = \sum_{L \in \calC_p(K)} \boldsymbol{\eta}(L), \qquad K \in
  \calC_{\mathrm{dyadic}}([0, 1]^d).
  \]
  We estimate
  \[
  \lvert \btheta_p(K) - \btheta_{p+1}(K) \rvert \leq \sum_{L \in
    \calC_p(K)} \left\lvert \boldsymbol{\eta}(L) - \sum_{L' \text{ children of }
    L} \boldsymbol{\eta}(L') \right\rvert \leq 2^{- pd \varepsilon} \kappa
  \lvert K \rvert^{1 + \varepsilon}.
  \]
  Consequently, $(\btheta_p)$ converges pointwise to a function
  $\btheta \colon \calC_{\mathrm{dyadic}}([0, 1]^d) \to \R$. Summing
  over $p$ in the preceding inequality yields
  \[
  \left\lvert \boldsymbol{\eta}(K) - \btheta(K) \right\rvert \leq
  \frac{\kappa}{1 - 2^{-d\varepsilon}} \lvert K
  \rvert^{1+\varepsilon}.
  \]
  In particular,
  \[
  |\btheta(K)| \leq |\boldsymbol{\eta}(K)| + |\boldsymbol{\eta}(K) - \btheta(K)| \leq
  \kappa \left( 1 + \frac{1}{1 - 2^{-d\varepsilon}} \right) |K|^\gamma.
  \]
  It remains to show that $\btheta$ is finitely additive, in the sense
  of Lemma~\ref{prop:gammaHolder=>strong}. Note that, for any $p \geq
  1$, one has
  \[
  \btheta_p(K) = \sum_{L \text{ children of } K} \btheta_{p-1}(L)
  \]
  We obtained the desired conclusion by letting $p \to \infty$. By
  Lemma~\ref{prop:gammaHolder=>strong}, we can now extend $\btheta$ to
  a $\gamma$-Hölder charge.
\end{proof}

If $f \in C^\beta([0, 1]^d)$ and $\bomega \in \rmsch^\gamma([0, 1]^d)$
are given and $\beta + d \gamma > d$, we can construct
$\boldsymbol{\eta} \colon \calC_{\mathrm{dyadic}}([0, 1]^d) \to \R$ by
picking in any dyadic cube $K$ a point $x_K \in K$ and defining
$\boldsymbol{\eta}(K) = f(x_K) \bomega(K)$. Let us check that (A) and
(B) in Theorem~\ref{thm:sewing} are satisfied. First,
\begin{align*}
\left\lvert \boldsymbol{\eta}(K) - \sum_{L \text{ children of } K}
\boldsymbol{\eta}(L) \right\rvert & \leq \sum_{L \text{ children of }
  K} |f(x_K) - f(x_L)| \lvert \bomega(L) \rvert \\ & \leq 2^d
\rmLip^\beta(f) (\operatorname{diam} K)^\beta \|\bomega\|_\gamma
|K|^\gamma \\ & \leq 2^d \rmLip^\beta(f) d^{\beta/2}
\|\bomega\|_\gamma |K|^{\gamma + \frac{\beta}{d}}.
\end{align*}
As for (B), one has simply $|\boldsymbol{\eta}(K)| \leq \|f\|_\infty
|K|^\gamma$. Thus we can apply the charge sewing lemma to
$\boldsymbol{\eta}$ and obtain by this way a charge $\btheta \in
\rmsch^\gamma([0, 1]^d)$. This charge coincides with our previous
definition of the Young indefinite integral $f \cdot \bomega$ by the
Young-Loeve estimate (Theorem~\ref{thm:youngLoeve}) and uniqueness.

\section{Connection with the Pfeffer Integral}
\label{sec:pfeffer}

\subsection{Regularity of a $BV$-set}
  We already introduced the isoperimetric coefficient of a $BV$-set to
  measure its roundedness in~\ref{e:BVsets}. In Pfeffer's integration
  theory, a more prominent role is played by the so-called
  \emph{regularity}
  \[
  \reg A = \frac{\lvert A\rvert}{\|A\| \diam A}
  \]
  defined for non-negligible $BV$-sets $A$.  Additionally, we set
  $\reg A = 0$ whenever $A$ is negligible. A $BV$-set $A$ is said to
  be \emph{$\varepsilon$-regular}, for some $\varepsilon > 0$, if
  $\reg A > \varepsilon$. Our definition of $\reg A$ differs slightly
  from Pfeffer's, where the diameter is calculated using the supremum
  norm. This change is innocuous (for instance, it does not affect the
  definition of Pfeffer's integral).

  The regularity and the isoperimetric coefficient of a $BV$-set $A$
  are related by the inequality
  \begin{equation}
    \label{eq:regIsop}
  \operatorname{reg} A \leq C \operatorname{isop} A
  \end{equation}
  for some constant $C$ depending only on $d$. Indeed, $A$ is
  contained in a closed ball of the same diameter as $A$, and
  therefore, $\lvert A \rvert \leq C(\operatorname{diam} A)^d$, where
  $C$ denotes here the volume of the unit closed ball. Then
  \[
  \operatorname{reg} A = \frac{\lvert A\rvert^{(d-1)/d} \lvert A
    \rvert^{1/d}}{\|A\| \diam A} = \operatorname{isop} A \frac{\lvert
    A \rvert^{1/d}}{\operatorname{diam} A} \leq C \operatorname{isop}
  A.
  \]
  However, controlling the regularity of a $BV$-set provides better
  information than the isoperimetric coefficient, as
  Lemma~\ref{lemma:regDiam} expresses that the volume of a $BV$-set is
  then comparable to the $d$-th power of its diameter. A control on
  the isoperimetric coefficient would not yield such a result
  (consider for example the union of two balls separated far away).

\begin{Lemma}
  \label{lemma:regDiam}
  For all $\varepsilon > 0$, there is a constant
  $C_{\ref{lemma:regDiam}}(\varepsilon) > 0$ such that
  \[
  \frac{1}{C_{\ref{lemma:regDiam}}(\varepsilon)} (\operatorname{diam}
  B)^d \leq \lvert B \rvert \leq C_{\ref{lemma:regDiam}}(\varepsilon)
  (\operatorname{diam} B)^d 
  \]
  for all $\varepsilon$-regular $BV$-sets $B$.
\end{Lemma}

\begin{proof}
  Actually, the constant in the upper bound can be taken independent
  from $\varepsilon$, the argument has been given in the preceding
  paragraph.

  To prove the lower bound, we use the isoperimetric inequality
  $\lvert A \rvert^{(d-1)/d} \leq C \|A\|$, which, when plugged in the
  expression for $\operatorname{reg} A$, gives
  \[
  \varepsilon < \operatorname{reg} A = \frac{\lvert A \rvert^{(d-1)/d}
    \lvert A \rvert^{1/d}}{\|A\| \operatorname{diam} A} \leq C
  \frac{\lvert A \rvert^{1/d}}{\operatorname{diam} A}
  \]
  therefore
  \[
  \lvert A \rvert \geq \left( \frac{\varepsilon}{C} \right)^d
  (\operatorname{diam} A)^d. 
  \]
\end{proof}

\subsection{Definitions}
  Let $A$ be a measurable set in $\R^d$. The \emph{essential closure}
  of $A$ is the set of points at which $A$ has a positive upper
  density, namely:
  \[
  \cless A = \left\{ x \in \R^d : \limsup_{r \to 0} \frac{\lvert A
    \cap B(x, r) \rvert}{r^d} > 0 \right\}.
  \]
  A set $A$ is said to be \emph{essentially closed} whenever $A =
  \cless A$. Requiring a set to be essentially closed is not a strong
  restriction, since any measurable set $A$ coincides a.e with its
  essential closure $\cless A$, which is essentially closed.

  A set $N \subset \R^d$ is called \emph{thin} whenever it has
  $\sigma$-finite $(d-1)$-dimensional Hausdorff measure. We do not
  require the measurability of $N$. A \emph{gage} on an essential
  closed $BV$-set $A$ is a nonnegative function $\delta \colon A \to
  \R_{\geq 0}$ such that $\{x \in A : \delta(x) = 0\}$ is thin.

  A \emph{tagged $BV$-partition} of $A$ is a finite set $\calP =
  \{(B_1, x_1), \dots, (B_p, x_p)\}$, where $B_1, \dots, B_p$ are
  pairwise disjoint elements of $\niceBV(A)$, and $x_k \in \cless B_k$
  for all indices $k$. The residual set $A \setminus \bigcup_{k=1}^p
  B_k$ might be non-empty; if it is empty, we say that $\calP$ is a
  \emph{tagged $BV$-division}.

  Let $\delta$ be a gage and $\varepsilon > 0$. A tagged
  $BV$-partition $\calP$ is said to be:
  \begin{itemize}
  \item \emph{$\delta$-fine} whenever $\operatorname{diam} B <
    \delta(x)$ for all $(B, x) \in \calP$;
  \item \emph{$\varepsilon$-regular} whenever $B$ is
    $\varepsilon$-regular for all $(B, x) \in \calP$.
  \end{itemize}
  Observe that, if $\delta(x) = 0$, then $x$ cannot appear as a tag in
  a $\delta$-fine partition. When $\delta > 0$ is a positive number,
  the expression ``$\calP$ is $\delta$-fine'' should be understood as
  ``$\calP$ is fine relative to the constant gage set to $\delta$''.

\subsection{Pfeffer integral}
  Let $A$ be an essentially closed $BV$-set, $\bomega \in
  \operatorname{ch}(A)$. A function $f \colon A \to \R$ is called
  \emph{Pfeffer-integrable} if there is a charge $\btheta \in
  \operatorname{ch}(A)$ such that for all $\varepsilon > 0$, there is
  a gage $\delta$ on $A$ such that
  \[
  \sum_{(B, x) \in \calP} \left\lvert \btheta(B) - f(x) \bomega(B)
  \right\rvert < \varepsilon
  \]
  for all $\delta$-fine $\varepsilon$-regular tagged $BV$-partition
  $\calP$ of $A$. In such a case, the charge $\btheta$ is unique, a
  fact that is not at all obvious and relies on a powerful Cousin
  lemma that provides the existence of $\delta$-fine
  $\varepsilon$-regular tagged $BV$-partition covering nearly all $A$
  (for values of $\varepsilon > 0$ below a constant), see~\cite{Pfef3,
    Pfef4} for a full treatment of the Pfeffer integral with respect
  to charges. This charge $\btheta$ is called the \emph{Pfeffer
    indefinite integral} of $f$ with respect to $\bomega$.
  
  It is clear from the definition that the restriction $f_{\mid B}$ of
  a function $f \in \operatorname{Pf}(A; \bomega)$ to an essentially
  closed $BV$-subset $B \subset A$ belongs to $\operatorname{Pf}(B;
  \bomega_{\mid B})$ (note that we do not assert that $f \ind_B \in
  \operatorname{Pf}(A; \bomega)$).

Historically, regular partitions were introduced in order to obtain a
generalized Gauss-Green theorem, allowing the integration of the
divergence of a merely differentiable vector field, first
by~\cite{Mawh}. In the entirely different context of integration with
respect to Hölder charges, the regularity of partitions also proves to
be crucial in obtaining the following result, which loosely expresses
that the Pfeffer integral encompasses the Young integral.

\begin{Theorem}
  Let $f \in C^\beta([0, 1]^d)$ and $\bomega \in \rmsch^\gamma([0, 1]^d)$,
  with $\beta + d \gamma > d$. Then $f$ is Pfeffer-integrable and its
  Pfeffer indefinite integral is equal to the Young indefinite
  integral
\end{Theorem}

\begin{proof}
  Let $\varepsilon > 0$. We will prove that $f$ is Pfeffer-integrable,
  and, in addition to that, it is possible to select a uniform
  gage. Let $\delta > 0$, to be determined later. Let $\calP$ be a
  $\delta$-fine $\varepsilon$-regular tagged $BV$-partition of $[0,
    1]^d$. For any $(B, x) \in \calP$, the set $B \cup \{x\}$ has the
  same diameter, perimeter and volume as $B$; and by the Young-Loeve
  estimate given in Theorem~\ref{thm:youngLoeve}, we obtain
  \[
  \left\lvert \int_B f \,  \bomega - f(x) \bomega(B) \right\rvert
  \leq C_{\ref{thm:youngLoeve}} \frac{\rmLip^\beta(f) \|\bomega\|_\gamma
    \lvert B \rvert^\gamma (\operatorname{diam}
    B)^\beta}{(\operatorname{isop} B)^{d(1-\gamma)}}.
  \]
  By~\eqref{eq:regIsop} and Lemma~\ref{lemma:regDiam}, we deduce
  \[
  \left\lvert \int_B f \,  \bomega - f(x) \bomega(B) \right\rvert
  \leq C(\varepsilon) \rmLip^\beta(f) \|\bomega\|_\gamma \lvert B
  \rvert^{\gamma+\beta/d}.
  \]
  By Lemma~\ref{lemma:regDiam} again,
  \begin{align*}
  \left\lvert \int_B f \,  \bomega - f(x) \bomega(B) \right\rvert &
  \leq C(\varepsilon) \rmLip^\beta(f) \|\bomega\|_\gamma \lvert B \rvert
  (\operatorname{diam} B)^{d\gamma+\beta - d} \\ & \leq C(\varepsilon)
  \rmLip^\beta(f) \|\bomega\|_\gamma \lvert B \rvert \delta^{d\gamma+\beta -
    d}.
  \end{align*}
  Summing over $\calP$ gives
  \[
  \sum_{(B, x) \in \calP} \left\lvert \int_B f \,  \bomega -
  f(x) \bomega(B) \right\rvert \leq C(\varepsilon) \rmLip^\beta(f)
  \|\bomega\|_\gamma \delta^{d\gamma+\beta - d}
  \]
  which can be made less than $\varepsilon$.
\end{proof}

Thanks to the preceding result, and because the Pfeffer integrability
of $f$ can be witnessed by constant gages, Young integrals can be
approximated by Riemann-Stieltjes sums, that is, 
\[
\int_{[0, 1]^d} f \,  \bomega = \lim_{n \to \infty} \sum_{(B, x) \in
  \calP_n} f(x) \bomega(B)
\]
if $(\calP_n)$ is a sequence of tagged $BV$-divisions of $[0, 1]^d$
that are $\varepsilon$-regular, for a fixed $\varepsilon > 0$, and
such that the meshes $\sup \{\operatorname{diam} B : (B, x) \in
\calP_n\}$ tend to $0$ as $n \to \infty$. (Actually, much more general
results are available within the setting of Pfeffer integration). For
instance, in~\cite{Zust}, the integral $\int_{[0, 1]^d} f \mathrm{d}
g_1 \wedge \cdots \wedge \mathrm{d} g_d$ was defined as the limit of
Riemann-Stieltjes sums, with the $\calP_n$ being the division by
dyadic cubes of generation $n$, tagged by their barycenters. In
\cite{AlbeStepTrev}, a similar Young integral is defined, this time
using dyadic decompositions of triangles.

In our approach, we only needed constant gages. It is worth noting the
work presented in \cite{BoonSeng}, which extensively uses variable
gages in the context of Henstock-Kurzweil integration in dimension one,
in order to make sense of integrals $\int_0^1 f \, \mathrm{d}g$, where
$f$ (resp. $g$) is a function of bounded $p$-variation
(resp. $q$-variation), $1/p + 1/q > 1$, including the case where $f$
and $g$ have common discontinuities.

There are numerous benefits in considering the Young integral, as
defined in this paper, as a particular case of the Pfeffer
integral. This allows us to take advantage of all the good properties
it possesses. The class of Pfeffer-integrable functions is highly
flexible; now that we know that $\beta$-Hölder continuous functions
are within it, it also includes all ``$\beta$-Hölder continuous
functions by $BV$-pieces''. Some properties, such as locality, which
are easy to prove for the Pfeffer integral, also carry over to the
Young integral, thanks to the previous result.

\begin{Proposition}[Locality]
  \label{prop:locality}
  Let $f \in C^\beta([0, 1]^d)$, $\bomega \in \rmsch^\gamma([0,
    1]^d)$, with $\beta + d \gamma > d$ and $B \in \niceBV([0,
    1]^d)$. If $f = 0$ on $B$ or $\bomega_{\mid B} = 0$, then $\int_B
  f \, \bomega = 0$.
\end{Proposition}

\subsection{Extension of Hölder charges}
  We conclude this section with an open problem. Let us consider a
  $BV$-set $A$. Since $A$ is bounded, we will assume that $A \subset
  [0, 1]^d$ without restricting the scope of our inquiry. We know that
  it is possible to extend any $\beta$-Hölder continuous function $f
  \colon A \to \R$ to a function $\tilde{f} \colon [0, 1] \to \R$ that
  is also $\beta$-Hölder continuous. This can be done by considering
  the McShane extension, as discussed in the proof of
  Lemma~\ref{lemma:LipDense} (with the additional property that
  $\rmLip^\beta f = \rmLip^\beta \tilde{f}$). Is it possible to extend
  a $\gamma$-Hölder charge $\bomega$ on $A$ to a charge
  $\tilde{\bomega} \in \rmsch^\gamma([0, 1]^d)$ (meaning that
  $\tilde{\bomega}_{\mid A} = \bomega$)?

  If this question has a positive answer, then it is possible to give
  meaning to the integral $\int_B f \, \bomega$ for any $B \in
  \niceBV(A)$ by defining it as:
  \[
  \int_B f \, \bomega = \int_B \tilde{f} \, \tilde{\bomega}.
  \]  
  Indeed, according to Proposition~\ref{prop:locality}, this quantity
  does not depend on the choice of extensions $\tilde{f}$ and
  $\tilde{\mu}$.

  \section{Hölder Differential Forms}
\label{sec:holderDF}

In this section, we will see how to associate a charge, denoted
$\diff g$ or $\diff g_1 \wedge \cdots \wedge \diff g_d$, to a
tuple of Hölder continuous functions $g = (g_1, \dots, g_d)$ on the
cube $[0, 1]^d$. The construction will make sense as long as the
Hölder exponents $\gamma_1, \dots, \gamma_d$ of $g_1, \dots, g_d$
satisfy the condition $\gamma_1 + \cdots + \gamma_d > d-1$, and it
produces a Hölder charge of exponent $(\gamma_1 + \cdots +
\gamma_d)/d$, that can be thought of as a generalized differential
form. Using the Young integral of the preceding section, it is now
possible to integrate
\[
\int_B f \diff g_1 \wedge \cdots \wedge \diff g_d
\]
over any $BV$-set $B \in \niceBV([0, 1]^d)$, if $f \in C^\beta([0,
  1]^d)$ and $\beta + \gamma_1 + \cdots + \gamma_d > d$. R. Züst not
only gave meaning to this type of integral when $B = [0, 1]^d$, but
also demonstrated that the condition regarding the Hölder exponents is
sharp (see \cite{Zust}, particularly section 3.2). His results are now
fully encompassed within the scope of the Young integral.  We will need
the following lemma.

\begin{Lemma}
  \label{lemma:LipDense}
  For any $g \in C^\beta([0, 1]^d)$, there is a sequence $(g_p)$ of
  functions $[0, 1]^d \to \R$ of class $C^\infty$ that converges
  uniformly to $g$, with $\rmLip^\beta(g_p) \leq \rmLip^\beta(g)$ for all
  $p$.
\end{Lemma}

\begin{proof}
  Extend $g$ to a $\beta$-Hölder continuous function on $\R^d$ while
  preserving the Hölder constant, for example with the McShane
  extension
  \[
  g(x) = \min_{y \in [0, 1]^d} \left( g(y) + \rmLip^\beta(g) \lvert y -
  x \rvert^\beta \right)
  \]
  and then regularize $g$ by convolution.
\end{proof}

\begin{Theorem}
  \label{thm:zust}
  Let $\gamma_1, \dots, \gamma_d \in (0, 1)$ such that
  \[
  \gamma = \frac{\gamma_1 + \cdots + \gamma_d}{d} > \frac{d-1}{d}.
  \]
  There is a unique multilinear operator
  \[
  \diff \colon C^{\gamma_1}([0, 1]^d) \times \cdots \times
  C^{\gamma_d}([0, 1]^d) \to \rmsch^\gamma([0, 1]^d)
  \]
  that satisfies the properties:
  \begin{itemize}
  \item[(A)] In case $g = (g_1,\dots, g_d)$ is $C^\infty$, then $\diff
    g = \det Dg \diff \calL$, where $Dg$ denotes the Jacobian matrix
    of $g$;
  \item[(B)] If $(g_{p,i})_{i \geq 0}$ are sequences of
    $\gamma_i$-Hölder continuous functions, converging uniformly to
    $g_i$, with $\rmLip^{\gamma_i}(g_{p,i})$ bounded, for all $i=1,
    \dots, d$, we have
    \[
    \lim_{p \to \infty} \left\| \diff g_{1} \wedge \cdots \wedge \diff
    g_{d} - \diff g_{p,1} \wedge \cdots \wedge \diff g_{p,d} \right\|
    = 0
    \]
  \end{itemize}
  Moreover,
  \[
  \left\| \diff g_{1} \wedge \cdots \wedge \diff g_{d} \right\|_\gamma
  \leq C_{\ref{thm:zust}} \prod_{i=1}^d \rmLip^{\gamma_i}(g_i)
  \]
  for all $g = (g_1, \dots, g_d)$.
\end{Theorem}

\begin{proof}
  The uniqueness follows clearly from (A), (B) and
  Lemma~\ref{lemma:LipDense}. We will prove the existence by induction
  on $d$. For the base case $d = 1$, the charge $\diff g_1$ is the
  one associated to the function $g_1 - g_1(0)$, that is, the only
  charge such that $\diff g_1([0, x]) = g_1(x) - g_1(0)$ for all
  $x \in [0, 1]^d$. As a charge, $\diff g_1$ is $\gamma_1$-Hölder
  by the discussion in~\ref{e:chargesdim1}. The properties (A) and (B)
  are easily satisfied (use Proposition~\ref{prop:weak} for (B)).

  Suppose $d \geq 2$, that the operator $\diff$ is already
  built in dimension $d-1$ and that
  \begin{equation}
    \label{eq:multiC}
    \|\diff \tilde{g}\|_{\tilde{\gamma}} = \|\diff
    \tilde{g}_2 \wedge \cdots \wedge \diff
    \tilde{g}_d\|_{\tilde{\gamma}} \leq C \prod_{i=2}^d
    \rmLip^{\gamma_i}(\tilde{g}_i)
  \end{equation}
  for some constant $C$ (that depends on $d$, among other things),
  where
  \[
  \tilde{\gamma} = \frac{\gamma_2 + \cdots + \gamma_d}{d-1}
  \]
  and $\tilde{g} = (\tilde{g}_2, \dots, \tilde{g}_d)$ is any tuple of
  functions $[0, 1]^{d-1} \to \R$ that are Hölder continuous functions
  with exponent $\gamma_2, \dots, \gamma_d$.

  For any dyadic cube $K = \prod_{i=1}^d [a_i, b_i]$ and $g = (g_1,
  \dots, g_d)$, we define
  \[
  \diff g_1 \wedge \cdots \wedge \diff g_d (K) =
  \int_{\partial K} g_1 \diff g_2 \wedge \cdots \wedge \diff
  g_d.
  \]
  The right-hand side should be interpreted as follows: we write
  $\tilde{K}_i$ for the dyadic cube in $[0, 1]^{d-1}$ that is the
  image of the cube $K$ under the projection map $(x_1, \dots, x_d)
  \mapsto (x_1, \dots, x_{i-1}, x_{i+1}, \dots, x_d)$ and we define
  the inclusions $[0, 1]^{d-1} \to [0, 1]^d$
  \begin{gather*}
    \iota_{i, 0} \colon (x_1, \dots, x_{i-1}, x_{i+1}, \dots, x_d) \mapsto (x_1, \dots, x_{i-1}, a_i,
    x_{i+1}, \dots, x_d) \\
    \iota_{i, 1} \colon (x_1, \dots, x_{i-1}, x_{i+1}, \dots, x_d) \mapsto (x_1, \dots, x_{i-1}, b_i,
  x_{i+1}, \dots, x_d)
  \end{gather*}
  then we define
  \[
  \int_{\partial K} g_1 \diff g_2 \wedge \cdots \wedge \diff
  g_d = \sum_{i=1}^d \sum_{j=0}^1 (-1)^{i+j} \int_{\tilde{K}_i} g_1
  \circ \iota_{i,j} \diff (g_2 \circ \iota_{i,j}) \wedge \cdots
  \wedge \diff (g_d \circ \iota_{i,j})
  \]
  which makes sense as a Young integral since $\gamma_1 + (d-1)
  \tilde{\gamma} = d \gamma > d-1$. For the moment, $\diff g_1
  \wedge \cdots \wedge \diff g_d$ is only defined on dyadic
  cubes, and by our sign convention,
  \[
  \diff g_1
  \wedge \cdots \wedge \mathrm{d} g_d(K) = \sum_{L \text{ children of } K} \mathrm{d}g_1
  \wedge \cdots \wedge \mathrm{d} g_d(L)
  \]
  for any dyadic cube $K$. Therefore, $\mathrm{d}g_1 \wedge \cdots
  \wedge \mathrm{d} g_d$ admits an additive extension to
  $\calF_{\mathrm{dyadic}}([0, 1]^d)$.

  Next we claim that $\mathrm{d}g_1 \wedge \cdots \wedge \mathrm{d}
  g_d(K) = 0$ if $g_1$ is constant. By (A) and the Stokes formula,
  this is true if $g_2, \dots, g_d$ are $C^\infty$. We then argue by
  density, using (B), Lemma~\ref{lemma:LipDense}
  and~\eqref{eq:multiC}.

  Therefore, if $g_1$ is arbitrary,
  \[
  \mathrm{d}g_1 \wedge \cdots \wedge \mathrm{d} g_d(K) =
  \mathrm{d}(g_1 - g_1(x)) \wedge \cdots \wedge \mathrm{d} g_d(K)
  \]
  where we choose $x \in K$. Note that the volumes of the cubes
  $\tilde{K}_i$ are related to that of $K$ by $|\tilde{K}_i = \lvert K
  \rvert^{(d-1)/d}$. We then use~\eqref{eq:indefboundK}
  and~\eqref{eq:multiC} in
  \begin{multline*}
  \left\lvert \int_{\tilde{K}_i} (g_1 - g_1(x)) \circ \iota_{i,j}
  \mathrm{d} (g_2 \circ \iota_{i,j}) \wedge \cdots \wedge \mathrm{d}
  (g_d \circ \iota_{i,j})\right\rvert \\ \leq C \left(\max_{y \in K}
  \lvert g_1(y) - g_1(x) \rvert + \lvert K \rvert^{\gamma_1 / d}
  \rmLip^{\gamma_1}(g_1) \right) \left( \prod_{i=2}^d \rmLip^{\gamma_i}(g_i)
  \right) \lvert K \rvert^{\frac{(d-1)\tilde{\gamma}}{d}}
  \end{multline*}
  Because $g_1$ is $\gamma_1$-Hölder continuous and the diameter of
  $K$ is comparable to $\lvert K \rvert^{1/d}$, we can bound $\max_{y
    \in K} |g_1(y) - g(x)|$ by a multiple of $|K|^{\gamma_1/d}
  \rmLip^{\gamma_1}(g_1)$. Summarizing, we have
  \[
  \lvert \mathrm{d}g(K) \rvert \leq C \left( \prod_{i=2}^d
  \rmLip^{\gamma_i}(g_i) \right) |K|^{\frac{\gamma_1}{d} + \frac{(d-1)
      \tilde{\gamma}}{d}} = C \left( \prod_{i=1}^d
  \rmLip^{\gamma_i}(g_i) \right) |K|^\gamma.
  \]
  We can now apply Lemma~\ref{prop:gammaHolder=>strong} and
  extend $\mathrm{d}g$ to a $\gamma$-Hölder charge, and moreover
  \begin{equation}
    \label{eq:multiC2}
    \|\mathrm{d}g \|_\gamma \leq C \prod_{i=1}^d \rmLip^{\gamma_i}(g_i).
  \end{equation}
  This ends the induction.

  Now that $\mathrm{d}$ is defined, we prove (A) by observing that
  $\mathrm{d}g$ and $\det Dg \diff \calL$ coincide and dyadic cubes,
  whenever $g$ is $C^\infty$. This can be proven by induction on $d$,
  with the Stokes formula. (B) is also proven by induction on $d$,
  using~\eqref{eq:multiC2} and Theorem~\ref{thm:youngIndefinite}(B).
\end{proof}

\section{Duality with Functions of Fractional Sobolev Regularity}
\label{sec:BValpha}

\subsection{Different seminorms}
In all of this section, we fix $\gamma \in \left(\frac{d-1}{d}, 1
\right)$ and we set $\alpha = d\gamma - (d - 1)$. Note that
necessarily $\alpha \in (0, 1)$.

  To a function $f \in L^1([0, 1]^d)$, we might associate several
  reasonable norms (with values in $[0, \infty]$):
  \begin{itemize}
  \item its \emph{Gagliardo} $W^{1-\alpha, 1}$ norm
    \[
    \|f\|_{W^{1-\alpha,1}} = \left|\int f\right| + \int \int \frac{|f(x)
      - f(y)|}{|x-y|^{d+1-\alpha}} \diff x \diff y
    \]
    For an introduction to fractional Sobolev spaces, the reader is
    referred to \cite{DiNePalaVald}. It is straightforward to check
    that sufficiently regular Hölder functions are fractional Sobolev,
    namely $\| f\|_{W^{1-\alpha,1}} < \infty$ if $f \in C^\beta([0,
      1]^d)$ and $\beta$ satisfies the Young type condition $\beta + d
    \gamma > d$ that was seen in Theorems~\ref{thm:young}
    and~\ref{thm:youngIndefinite}.
  \item its \emph{Züst fractional $\alpha$-variation}
    \[
    \V^\alpha f = \sup \int f \det Dg
    \]
    where the supremum is taken over all Lipschitz maps $g = (g_1,
    \dots, g_d)$ on $[0, 1]^d$ with $\rmLip^\alpha g_1 \leq 1$ and
    $\rmLip g_2, \dots, \rmLip g_d \leq 1$.  Functions of fractional
    bounded variation in this sense were introduced in the paper
    \cite{Zust2}.
  \item its \emph{$\tilde{\V}^\alpha$ variation}
    \[
    \tilde{\V}^\alpha f = \sup \left\{ \int f \operatorname{div}
    \varv : \varv \in C^1([0, 1]^d; \R^d) \text{ and } \rmLip^\alpha
    \varv \leq 1 \right\}
    \]
  \item its \emph{$\hat{\V}^\alpha$ variation}
    \[
    \hat{\V}^\alpha f = \sup \left\{ \int fg : g \in L^{\infty}([0,
      1]^d) \text{ and } \| g \diff \calL \|_\gamma \leq 1 \right\}
    \]
  \item finally, we define the $\ell^1$ type norm
    \[
    \bN^\alpha f = \left| \int f \right| + \sum_{n=0}^\infty
    \sum_{k=0}^{2^{nd}-1} \sum_{r=1}^{2^d-1} |b_{n,k,r}(f)|
    \]
    where, for all relevant indices $n, k, r$, we defined the
    renormalized Haar coefficient
    \[
    b_{n,k,r}(f) = \frac{1}{2^{nd\left( \gamma - \frac{1}{2} \right)}} \int f g_{n,k,r}.
    \]
  \end{itemize}
  We will show that all those norms are equivalent (the equivalence
  between $\V^\alpha$ and $\hat{\V}^\alpha$ was stated as an open
  question in \cite{Zust2}). They define the same space
  $W^{1-\alpha,1}((0, 1)^d) = \{ f \in L^1([0, 1]^d) :
  \|f\|_{W^{1-\alpha, 1}} < \infty \}$. We will then prove
  (Proposition~\ref{prop:dual}) that $\rmsch^\gamma([0, 1]^d)$ is the
  dual of $W^{1-\alpha, 1}((0, 1)^d)$.

  In the same paper \cite{Zust2}, R. Züst defined the space
  $BV^\alpha_c(\R^d)$ of compactly supported functions of
  $\alpha$-fractional variation. It easily follows from the results of
  this section that this space coincides with
  $W^{1-\alpha,1}_c(\R^d)$, the space of compactly supported
  fractional Sobolev functions, for the fractional exponent
  $1-\alpha$. This answers a question raised in \cite{ComiStef}: the
  class of functions of fractional bounded variation in the sense of
  Züst is strictly smaller than those in the sense of Comi and
  Stefani, as introduced in~\cite{ComiStef}.

\begin{Theorem}
  The norms $\| \cdot \|_{W^{1-\alpha,1}}$, $\V^\alpha$,
  $\tilde{\V}^\alpha$, $\hat{\V}^\alpha$ and $\bN^\alpha$ are
  equivalent.
\end{Theorem}

\begin{proof}[Proof that $\| \cdot \|_{W^{1-\alpha, 1}} \sim \bN^\alpha$]
  This result is already known. It amounts to saying that, up to
  renormalization, the Haar system is a Schauder basis of
  $W^{1-\alpha,1}((0, 1)^d)$ equivalent to the canonical basis of
  $\ell^1$. Indeed, by \cite[Appendix~A, Theorem~A.1]{BourBrezMiro},
  the $W^{1-\alpha,1}$ seminorm
  \[
  [f]_{W^{1-\alpha,1}} = \int \int \frac{|f(x) -
    f(y)|}{|x-y|^{d+1-\alpha}} \diff x \diff y
  \]
  is equivalent to
  \[
  [f]_{W^{1-\alpha,1}} \sim \sum_{n=0}^\infty 2^{n(1-\alpha)} \| E_{n+1} f - E_{n} f\|_1 
  \]
  where, for each integer $n$, we set $E_nf$ to be the projection of
  $f$ onto the space of functions constant on dyadic cubes of
  generation $n$, so that
  \[
  E_{n+1}f - E_nf = 2^{nd\left( \gamma - \frac12
    \right)}\sum_{k=0}^{2^{nd}-1} \sum_{r=1}^{2^d-1} b_{n,k,r}(f)
  g_{n,k,r}
  \]
  Clearly,
  \begin{align*}
  \| E_{n+1}f - E_nf \|_1 & \leq 2^{nd\left( \gamma - \frac12
    \right)}\sum_{k=0}^{2^{nd}-1} \sum_{r=1}^{2^d-1} b_{n,k,r}(f) \|
  g_{n,k,r} \|_1 \\ & \leq 2^{nd(\gamma-1)}\sum_{k=0}^{2^{nd}-1}
  \sum_{r=1}^{2^d-1} |b_{n,k,r}(f)|
  \end{align*}
  and thus
  \[
  \sum_{n=0}^\infty 2^{n(1-\alpha)} \| E_{n+1} f - E_{n} f\|_1 \leq
  \sum_{n=0}^\infty \sum_{k=0}^{2^{nd}-1} \sum_{r=1}^{2^d-1}
  |b_{n,k,r}(f)|
  \]
  For the reverse inequality, we choose $r_n$ so that
  \[
  \sum_{k=0}^{2^{nd}-1} |b_{n,k,r_n}(f)| \geq \frac{1}{r}
  \sum_{k=0}^{2^{nd}-1} \sum_{r=1}^{2^d-1} |b_{n,k,r}(f)|
  \]
  Define
  \[
  g_n = \sum_{k=0}^{2^{nd}-1} \operatorname{sign}(b_{n,k,r_n}(f)) g_{n,k,r_n}
  \]
  so that
  \[
  \int (E_{n+1} f - E_n f) g_n = 2^{nd\left( \gamma - \frac12 \right)}
  |b_{n,k,r_n}(f)| \leq \| E_{n+1} f - E_n f\|_1 \|g_n\|_\infty
  \]
  Consequently,
  \[
  2^{n(1-\alpha)} \|E_{n+1}f - E_n\|_1 \geq \frac{1}{r}
  \sum_{k=0}^{2^{nd}-1} \sum_{r=1}^{2^d-1} |b_{n,k,r}(f)|
  \]
  We conclude by summing over $n$.
\end{proof}

\begin{proof}[Proof that $\tilde{\V}^\alpha \leq C \V^\alpha$]
  Let $f \in L^1([0, 1]^d)$ and $\varv \in C^1([0, 1]^d ; \R^d)$. Then
  \begin{align*}
    \int_{[0, 1]^d} f \operatorname{div} \varv & = \sum_{k=1}^d
    \int_{[0, 1]^d} f \frac{\partial \varv_k}{\partial x_k} \\ & =
    \sum_{k=1}^d (-1)^{k+1} \int_{[0, 1]^d} f \det D(\varv_k, x_1, \dots,
    \widehat{x_k}, \dots, x_d)
  \end{align*}
  where $x_1, \dots, x_d$ are the coordinate functions, and
  $\widehat{x_k}$ means that the function $x_k$ is removed from the
  above list. Thus,
  \[
  \int_{[0, 1]^d} f \operatorname{div} \varv \leq d \V^\alpha f
  \rmLip^\alpha \varv.
  \]
  This proves that $\tilde{\V}^\alpha f \leq d \V^\alpha f$, for all
  $f \in L^1([0, 1]^d)$.
\end{proof}

\begin{proof}[Proof that $\V^\alpha \leq C \hat{\V}^\alpha$]
  Let $f \in L^1([0, 1]^d)$ and $g = (g_1, \dots, g_d) \in \rmLip([0,
    1]^d; \R^d)$. By Theorem~\ref{thm:zust}, the charge $\diff g_1
  \wedge \cdots \wedge \diff g_d$ has the density $\det Dg$ in
  $L^\infty([0, 1]^d)$ and
  \[
  \| \diff g_1 \wedge \cdots \wedge \diff g_d \|_\gamma \leq
  C_{\ref{thm:zust}} \rmLip^\alpha g_1 \prod_{k=2}^d \rmLip g_k.
  \]
  This entails that $\V^\alpha f \leq C_{\ref{thm:zust}}
  \hat{\V}^\alpha f$.
\end{proof}

\begin{proof}[Proof that $\bN^\alpha \leq C \hat{\V}^\alpha$]
  Let $f \in L^1([0, 1]^d)$. For any relevant indices $n,k,r$, define
  $\varepsilon_{n,k,r} \in \{-1, 0, 1\}$ to be the sign of
  $b_{n,k,r}(f)$. Also set $\varepsilon_{-1}$ to the sign of $\int
  f$. We define the sequence $(g_p)$ in $L^\infty([0, 1]^d)$ by
  \[
  g_p = \varepsilon_{-1} + \sum_{n=0}^p \sum_{k=0}^{2^{nd}-1}
  \sum_{r=1}^{2^d-1} 2^{nd \left(  \frac{1}{2} - \gamma
    \right)}\varepsilon_{n,k,r} g_{n,k,r}.
  \]
  Since
  \[
  g_p \diff \calL = \varepsilon_{-1} \diff \calL + \sum_{n=0}^p
  \sum_{k=0}^{2^{nd}-1} \sum_{r=1}^{2^d-1} 2^{nd \left( \frac{1}{2} -
    \gamma \right)} \varepsilon_{n,k,r} \bomega_{n,k,r}
  \]
  Proposition~\ref{lemma:decay} entails that $\|g_p \diff
  \calL\|_\gamma \leq C_{\ref{lemma:decay}}$ for all integers
  $p$. However, we have
  \[
  \int fg_p = \left\lvert \int f \right\rvert + \sum_{n=0}^p
  \sum_{k=0}^{2^{nd}-1} \sum_{r=1}^{2^d-1} \lvert b_{n,k,r}(f) \rvert
  \leq \hat{\V}^\alpha f \|g_p \diff \calL\|_\gamma \leq
  C_{\ref{lemma:decay}} \hat{\V}^\alpha f.
  \]
  We conclude by letting $p \to \infty$.
\end{proof}

\begin{proof}[Proof that $\hat{\V}^\alpha \leq C \bN^\alpha$]
  Let $f \in L^1([0, 1]^d)$ be a function, that we decompose along the Haar basis (in $L^1$)
  \[
  f = \int f + \sum_{n=0}^\infty \sum_{k=0}^{2^{nd}-1}
  \sum_{r=1}^{2^d-1} 2^{nd\left( \gamma- \frac{1}{2} \right)}
  b_{n,k,r}(f) g_{n,k,r}
  \]
  This means that $f_p \to f$ in $L^1$, if $f_p$ denotes the truncated series
  \[
  f_p = \int f + \sum_{n=0}^p \sum_{k=0}^{2^{nd}-1}
  \sum_{r=1}^{2^d-1} 2^{nd\left( \gamma- \frac{1}{2} \right)}
  b_{n,k,r}(f) g_{n,k,r}
  \]
  Likewise, we take $g \in L^\infty([0, 1]^d)$ and decompose the
  charge $g \diff \calL$ along the Faber-Schauder basis
  \[
  g \diff \calL = \left(\int g\right) \diff \calL + \sum_{n=0}^\infty \sum_{k=0}^{2^{nd}-1} \sum_{r=1}^{2^d
    - 1} a_{n,k,r}(g \diff \calL) \bomega_{n,k,r}.
  \]
  Then
  \begin{align*}
    \int fg & = \lim_{p \to \infty} \int f_p g \\ & = \lim_{p \to
      \infty} (g \diff \calL)(f_p) \\ & = \lim_{p \to \infty} \left(
    \int f \int g + \sum_{n=0}^p \sum_{k=0}^{2^{nd}-1}
    \sum_{r=1}^{2^d-1} a_{n,k,r}(g \diff \calL) 2^{nd\left( \gamma-
      \frac{1}{2} \right)} b_{n,k,r}(f) \right)
  \end{align*}
  By Proposition~\ref{lemma:decay} again, we obtain
  \[
  \int fg \leq C_{\ref{lemma:decay}} \left( \left\lvert \int
  f\right\rvert + \sum_{n=0}^\infty \sum_{k=0}^{2^{nd}-1}
  \sum_{r=1}^{2^d - 1} \lvert b_{n,k,r}(f) \rvert \right) \| g \diff
  \calL \|_\gamma.
  \]
  This concludes the proof.
\end{proof}

It remains to prove that $\hat{\V}^\alpha \leq C
\tilde{\V}^\alpha$. This last estimate is harder and relies on a
representation theorem for Hölder charges, that may be of independent
interest. We will need the material of the next subsection.

\subsection{Representation of a Hölder charge as the divergence of a Hölder vector field}
  \label{e:divalpha}
  In \cite[Theorem 4.5]{BouaDePa}, it was proven that any strong
  charge on $[0, 1]^d$ is the flux $\operatorname{div} \varv$ of a
  continuous vector field $\varv \colon [0, 1]^d \to \R^d$, building
  upon ideas present in \cite{BourBrez} and \cite{DePaPfef}, that is,
  of the form
  \[
  \operatorname{div} \varv \colon B \mapsto \int_{\partial^* B} \varv
  \cdot n_B \diff \calH^{d-1}, \qquad B \in \niceBV([0, 1]^d)
  \]
  where $\partial^*B$ is the reduced boundary of $B$, $\calH^{d-1}$ is
  the $(d-1)$-dimensional Hausdorff measure, and $n_B$ is the outer
  normal vector defined $\calH^{d-1}$-almost everywhere on
  $\partial^*B$ (see \cite{EvanGari}). Moreover, $\operatorname{div}
  \colon C([0, 1]^d; \R^d) \to \rmsch([0, 1]^d)$ is a continuous
  operator.

  It might have been preferable to represent strong charges as
  exterior differentials of continuous differential forms, in line with
  our intuition that strong charges are generalized differential
  forms. However, we have kept the representation in term of
  divergence, which amounts to identifying continuous $(d-1)$-forms with
  continuous vector fields.

  We make the following observation: if $\varv$ is $\alpha$-Hölder
  continuous, then, for any dyadic cube $K$, one may choose any point
  $x \in K$ and the compute
  \begin{align*}
    \lvert (\operatorname{div} \varv) (K) \rvert & \leq \left\lvert
    \int_{\partial K} \varv \cdot n_K \diff \calH^{d-1} \right\rvert
    \\ & \leq \left\lvert \int_{\partial K} (\varv - \varv(x)) \cdot
    n_K \diff \calH^{d-1} \right\rvert \\ & \leq \left(\rmLip^\alpha
    \varv \right) \calH^{d-1}(\partial K) (\operatorname{diam}
    K)^{\alpha} \\ & \leq C \left(\rmLip^\alpha \varv \right) \lvert K
    \rvert^{\frac{d-1}{d}} \, \lvert K \rvert^{\frac{d \gamma - d +
        1}{d}} \\ & \leq C \left( \rmLip^\alpha \varv\right) |K|^\gamma
  \end{align*}
  The constant $C$ may depend on $d$, $\gamma$, but not on $K$. Hence
  $\operatorname{div} \varv \in \rmsch^\gamma([0, 1]^d)$. Thus,
  $\operatorname{div}$ restricts to a well-defined operator
  \[
  \operatorname{div} \colon C^\alpha([0, 1]^d; \R^d) \to
  \rmsch^\gamma([0, 1]^d).
  \]
  If $(\varv_n)$ is a bounded sequence in $C^\alpha([0, 1]^d; \R^d)$
  that converges uniformly to $\varv$, then it is clear that
  $(\operatorname{div} \varv_n)(K) \to (\operatorname{div} \varv)(K)$
  for all dyadic cube $K$. By Proposition~\ref{prop:weak},
  $\operatorname{div} \varv_n \to \operatorname{div} \varv$ in
  $\rmsch([0, 1]^d)$.

\begin{Theorem}[Representation of Hölder charges]
  \label{thm:repr}
  The operator $\operatorname{div} \colon C^\alpha([0, 1]^d; \R^d) \to
  \rmsch^\gamma([0, 1]^d)$ is onto.
\end{Theorem}

\begin{proof}
  We fix $\beta \in (\alpha, 1)$.  For any $r \in \{1, \dots, 2^d -
  1\}$, the Haar function $g_{0,0,r}$ (of generation $0$) belongs to
  $L^\infty([0, 1]^d)$, and therefore to any $L^p([0, 1]^d)$, for $1
  \leq p < \infty$. Moreover, it has zero mean. Those two properties
  guarantee that we can find a weak solution $\varv_{0,0,r} \in W^{1,
    p}((0,1)^d)^d$ for the problem
  \[
  \operatorname{div} \varv_{0,0,r} = g_{0,0,r} \quad \text{ and } \quad
  \varv_{0,0,r} = 0 \text{ on } \partial[0, 1]^d.
  \]
  This follows from the work of Bourgain and Brézis, see
  \cite[Section~7.1, Theorem~2]{BourBrez}, where the boundary
  condition above is called the full Dirichlet condition. By selecting
  $p$ high enough, we can suppose that $\varv_{0,0,r} \in C^\beta([0,
    1]^d; \R^d)$, thanks to the Sobolev embeddings.

  That $\varv_{0, 0, r}$ is a weak solution means, in this context,
  that
  \[
  - \int_{[0, 1]^d} \varv \cdot \nabla \varphi = \int_{[0,1]^d}
  g_{0,0,r} \varphi \qquad \text{for all } \varphi \in C^1_c((0,1)^d).
  \]
  By easy approximation arguments, we can argue that
  \[
  - \int_{[0, 1]^d} \varv \cdot \diff Du = \int_{[0,1]^d}
  g_{0,0,r} u \qquad \text{for all } u \in BV([0, 1]^d)
  \]
  where $Du$ denotes the variation measure of $u$. In particular, if
  we apply the above result to $u = \ind_B$ for $B \in \niceBV([0,
    1]^d)$, we obtain that the charge $\bomega_{0,0,r}$ is the
  divergence of $\varv_{0,0,r}$.
  
  For any relevant indices $n,k, r$, one define the map $\varv_{n,k,r}
  \in C^\alpha([0, 1]^d; \R^d)$, with support in $K_{n,k}$, by setting
  $\varv_{n,k,r}(x) = 0$ if $x \not\in K_{n,k}$ and
  \[
  \varv_{n,k,r}(x) = 2^{n \left( \frac{d}{2} - 1
    \right)}\varv_{0,0,r}(2^n(x - x_{n,k,r})) \qquad \text{for all } x
  \in K_{n,k}
  \]
  where $x_{n,k,r} \in [0, 1]^d$ is the bottom-left point of
  $K_{n,k}$. We let the reader check that $\bomega_{n,k,r} =
  \operatorname{div} \varv_{n,k,r}$. The following estimates on
  $\varv_{n,k,r}$ will soon prove useful:
  \begin{equation}
    \label{eq:estimVNKR}
  \| \varv_{n,k,r} \|_\infty \leq C 2^{n\left( \frac{d}{2} - 1 \right)}, \qquad
  \rmLip^\beta \varv_{n,k,r} \leq C 2^{n\left( \beta + \frac{d}{2} - 1 \right)}
  \end{equation}
  for all relevant indices $n,k,r$, where $C$ is a constant.
  
  Let $\bomega \in \rmsch^\gamma([0, 1]^d)$ be a Hölder charge (for
  which we are seeking a pre-image by $\operatorname{div}$). Recall
  that it can be decomposed
  \[
  \bomega = \bomega([0, 1]^d) \diff \calL + \sum_{n= 0}^\infty
  \sum_{k=0}^{2^{nd}-1} \sum_{r=1}^{2^d-1} a_{n,k,r}(\bomega)
  \bomega_{n,k,r}.
  \]
  With no loss of generality, we will suppose that $\bomega([0, 1]^d)
  = 0$. This is possible since the Lebesgue charge $\diff\calL$ is the
  divergence of the vector field $x \mapsto (x_1, 0, \dots, 0)$, for
  example.

  For any nonnegative integer $n$, we define the vector field
  \[
  \varv_n = \sum_{k=0}^{2^{nd}-1} \sum_{r=1}^{2^d-1}
  a_{n,k,r}(\bomega) \varv_{n,k,r}.
  \]
  Next we use the fact that the support $K_{n,k}$ of the fields
  $\varv_{n,k,r}$ are pairwise almost disjoint (and $r$ is allowed to
  take a constant number of values), that $\varv_{n,k,r}$ vanishes on
  $\partial K_{n,k}$, and the estimates~\eqref{eq:estimVNKR} to infer
  that
  \begin{align*}
    \| \varv_n \|_\infty & \leq C \left(\max_{k,r}
    |a_{n,k,r}(\bomega)|\right) 2^{n \left( \frac{d}{2} - 1 \right)}
    \\ \rmLip^\beta \varv_n & \leq C \left(\max_{k,r}
    |a_{n,k,r}(\bomega)|\right) 2^{n \left( \beta + \frac{d}{2} - 1
      \right)}
  \end{align*}
  Taking Lemma~\ref{lemma:decay} into account,
  \begin{align}
    \| \varv_n \|_\infty & \leq C \| \bomega \|_\gamma 2^{-n \alpha} \label{eq:varvU}
    \\ \rmLip^\beta \varv_n & \leq C \| \bomega \|_\gamma 2^{n (\beta
      - \alpha)}
  \end{align}
  Thanks to~\eqref{eq:varvU}, the series $\sum_{n=0}^\infty \varv_n$
  converges uniformly to a vector field denoted $\varv$.

  We claim that
  \begin{equation}
    \label{eq:aqwzsx}
    \sum_{n=0}^\infty \lvert \varv_n(x) - \varv_n(y) \rvert \leq C
    \|\bomega\|_\gamma \lvert x - y \rvert^\alpha
  \end{equation}
  for all $x, y \in [0, 1]^d$, where $C \geq 0$ is a constant. Indeed,
  for each $p$, we have
  \begin{align*}
    \sum_{n=0}^\infty \lvert \varv_n(x) - \varv_n(y) \rvert & = \sum_{n=0}^p
    \lvert \varv_n(x) - \varv_n(y) \rvert + \sum_{n=p+1}^\infty \lvert \varv_n(x)
    - \varv_n(y) \rvert \\
    & \leq \sum_{n=0}^p \rmLip^\beta \varv_n |x - y|^\beta + 2 \sum_{n=p+1}^\infty \|v_n\|_\infty \\
    & \leq C \|\bomega\|_\gamma \left( \sum_{n=0}^p 2^{n(\beta - \alpha)} |x - y|^\beta + \sum_{n=p+1}^\infty 2^{-n\alpha}\right) \\
    & \leq C \|\bomega\|_\gamma \left( 2^{p(\beta - \alpha)} |x - y|^\beta + \frac{1}{2^{p\alpha}} \right)
  \end{align*}
  Choosing $p$ such that
  \[
  \frac{1}{2^{p+1}} \leq \frac{|x-y|}{\sqrt{d}}\leq \frac{1}{2^p}
  \]
  we finally obtain~\eqref{eq:aqwzsx}.

  In particular, \eqref{eq:aqwzsx} implies that the partial sums of
  the series $\sum \varv_n$ are uniformly $\alpha$-Hölder
  continuous. Thanks to the continuity property of
  $\operatorname{div}$ mentioned in~\ref{e:divalpha},
  \[
  \operatorname{div} \varv = \sum_{n=0}^\infty \operatorname{div}
  \varv_n = \bomega. 
  \]
\end{proof}

We remark that the above proof is constructive. Indeed, we have built
a linear continuous right inverse of $\operatorname{div} \colon
C^\alpha([0, 1]^d; \R^d) \to \rmsch^\gamma([0, 1]^d)$.  This contrasts
with the operator $\operatorname{div} \colon C ([0, 1]^d; \R^d) \to
\rmsch([0, 1]^d)$, which is also known to be onto. In
\cite[Theorem~4.4]{Boua}, it is proven that it does not have a
uniformly continuous (possibly nonlinear) right inverse.

\begin{Lemma}
  \label{lemma:approxV}
  For each $f \in L^1([0, 1]^d)$, there is a sequence $(f_n)$ in
  $C^\infty([0, 1]^d)$ that converges towards $f$ in $L^1([0, 1]^d)$
  and such that $\tilde{\V}^\alpha f_n \leq \tilde{\V}^\alpha f$ for
  all $n$.
\end{Lemma}

\begin{proof}
  Let us extend $f$ to $\R^d$ by setting $f = 0$ on $\R^d \setminus
  [0, 1]^d$. Let $\varphi \colon \R^d \to \R$ be a nonnegative
  $C^\infty$ function, that is supported in the unit ball of $\R^d$
  and such that $\int_{\R^d} \varphi = 1$. For all $\varepsilon > 0$,
  we set $\varphi_\varepsilon(x) = \frac{1}{\varepsilon^d}
  \varphi\left(\frac{x}{\varepsilon}\right)$.

  Note that $f * \varphi_\varepsilon$ is $C^\infty$ and compactly
  supported in $[-\varepsilon, 1+\varepsilon]^d$. Let
  $\Phi_\varepsilon \colon \R^d \to \R^d$ be the affine map $x \mapsto
  (1+2\varepsilon) x - (\varepsilon, \dots, \varepsilon)$ (that sends
  the cube $[0, 1]^d$ to $[-\varepsilon, 1+\varepsilon]^d$).

  We set
  \[
  f_\varepsilon = (f * \varphi_\varepsilon) \circ \Phi_\varepsilon
  \]
  Let $\varv \in C^1(\R^d; \R^d)$. Then
  \begin{align*}
    \int_{[0, 1]^d} f_\varepsilon(x) \operatorname{div} \varv(x) \diff
    x & = \int_{\R^d} f_\varepsilon(x) \operatorname{div} \varv(x)
    \diff x \\ & = \frac{1}{(1+2\varepsilon)^d} \int_{\R^d} f *
    \varphi_\varepsilon(x) \operatorname{div}
    \varv(\Phi_\varepsilon^{-1}(x)) \diff x \\ & =
    \frac{1}{(1+2\varepsilon)^{d-1}} \int_{\R^d} f *
    \varphi_\varepsilon(x) \operatorname{div} \varv \circ
    \Phi_\varepsilon^{-1}(x) \diff x \\ & = \frac{1}{(1 + 2
      \varepsilon)^{d-1}} \int_{\R^d} \left( \int_{\R^d} f(x-y)
    \operatorname{div} \varv \circ \Phi_\varepsilon^{-1}(x) \diff
    x\right) \varphi_\varepsilon(y) \diff y
  \end{align*}
  For each $y \in \R^d$ such that $|y| \leq \varepsilon$,
  \[
  \int_{\R^d} f(x-y) \operatorname{div} \varv \circ
  \Phi_\varepsilon^{-1}(x) \diff x = \int_{[0, 1]^d} f(x) \operatorname{div} \varv \circ \Phi_\varepsilon^{-1}(x+y) \diff x
  \]
  Since
  \[
  \rmLip^\alpha \left( x \in [0, 1]^d \mapsto \varv \circ
  \Phi_\varepsilon^{-1}(x+y) \right) \leq \frac{\rmLip^\alpha
    \varv_{\mid [0, 1]^d}}{(1+ 2 \varepsilon)^{\alpha}} \leq
  \rmLip^\alpha \varv_{\mid [0, 1]^d}.
  \]
  This implies that
  \[
  \int_{\R^d} f(x-y) \operatorname{div} \varv \circ
  \Phi_\varepsilon^{-1}(x) \diff x \leq  \tilde{\V}^\alpha f
  \]
  and in turn that
  \[
  \int_{[0, 1]^d} f_\varepsilon(x) \operatorname{div} \varv(x) \diff x
  \leq \tilde{\V}^\alpha f.
  \]
  As this holds for all $\varv$, we can conclude that
  $\tilde{\V}^\alpha(f * \varphi_\varepsilon) \leq \tilde{\V}^\alpha
  f$.

  It remains to show that $f_\varepsilon \to f$ in $L^1([0, 1]^d)$ as
  $\varepsilon \to 0$. The proof is carried out in a standard manner
  by density. It is observed, indeed, that if $f \in C(\R^d)$ has
  compact support in $[0, 1]^d$, the convergence is in fact uniform.
\end{proof}

\begin{proof}[Proof that $\hat{\V}^\alpha \leq C \tilde{\V}^\alpha$]
  First we will assume that $f \in BV([0, 1]^d)$.

  Since $\operatorname{div}$ is onto, by Theorem~\ref{thm:repr}, it
  induces a Banach space isomorphism
  \[
  \frac{C^\alpha([0, 1]^d); \R^d)}{\ker \operatorname{div}} \to
  \rmsch^\gamma([0, 1]^d).
  \]
  Hence the existence of a constant $C \geq 0$ such that, for all
  $\bomega \in \rmsch^\gamma([0, 1]^d)$, there is $\varv \in
  C^\alpha([0, 1]^d; \R^d)$ with $\bomega = \operatorname{div} \varv$
  and $\rmLip^\alpha \varv \leq C \|\bomega\|_\gamma$.

  We apply this result to $\bomega = g \diff \calL$, where $g \in
  L^\infty([0, 1]^d)$ and $\|g \diff \calL\|_\gamma \leq 1$. We apply
  Lemma~\ref{lemma:LipDense} to the coordinate functions $\varv_1,
  \dots, \varv_d$, and so we obtain $C^\infty$ approximating sequences
  $(\varv_{i, n})$ (for $1 \leq i \leq d$) that converge uniformly to
  $\varv_i$, with
  \[
  \rmLip^\alpha \varv_{i,n} \leq \rmLip^\alpha \varv_i \leq
  \rmLip^\alpha \varv \leq C \|g \diff \calL\|_\gamma \leq C
  \]
  for all $n$. The vector fields $\varv_{\bullet, n} = (\varv_{1, n},
  \dots, \varv_{d,n})$ are uniformly $\alpha$-Hölder continuous, with
  $\rmLip^\alpha \varv_{\bullet, n} \leq \sqrt{d} \rmLip^\alpha \varv$
  and converge uniformly towards $\varv$. By the discussion
  in~\ref{e:divalpha}, we deduce that $\operatorname{div}
  \varv_{\bullet, n} \to \operatorname{div} \varv$ in $\rmsch([0,
    1]^d)$. On the other hand, by the classical divergence theorem,
  the charge $\operatorname{div} \varv_{\bullet, n}$ has density given
  by the (pointwise) divergence $\operatorname{div} \varv_{\bullet,
    n}$. Expressed in terms of strong charge functionals, the limit
  $\operatorname{div} \varv_{\bullet, n} \to \operatorname{div} \varv$
  (in $\rmsch([0, 1]^d)$) can be rewritten $T_{\operatorname{div}
    \varv_{\bullet, n}} \to T_g$ (in $SCH([0, 1]^d)$). Hence, as $f
  \in BV([0, 1]^d)$, we can write
  \[
  \int_{[0, 1]^d} fg = \lim_{n \to \infty} \int_{[0, 1]^d} f
  \operatorname{div} \varv_{\bullet, n} \leq C \tilde{\V}^\alpha f.
  \]
  This proves that $\hat{\V}^\alpha f \leq C \tilde{\V}^\alpha f$.

  Now, suppose that $f \in L^1([0, 1]^d)$. There is an approximating
  sequence $(f_n)$ in $C^\infty([0, 1]^d)$ as in
  Lemma~\ref{lemma:approxV}. By what precedes, $\hat{\V}^\alpha f_n
  \leq C \tilde{\V}^\alpha f_n \leq C \tilde{\V}^\alpha f$. But
  $\hat{\V}^\alpha$ (as well as $\tilde{\V}^\alpha$) is clearly lower
  semicontinuous with respect to $L^1$-convergence. Letting $n \to
  \infty$, we deduce that $\hat{\V}^\alpha f \leq C \tilde{\V}^\alpha
  f$.
\end{proof}

\subsection{Duality}

Our final goal is to show that $\rmsch^\gamma([0, 1]^d)$ is the dual
space of $W^{1-\alpha,1}((0, 1)^d)$. Of course, this is true on an
abstract level, since the former space is isomorphic to $\ell^1$
(being normed by $\bN^\alpha$), and the latter to
$\ell^\infty$. However, the next proposition shows that this duality
can be naturally witnessed by a duality bracket $\langle \cdot, \cdot
\rangle$ that extends the Young integral of Theorem~\ref{thm:young}.

\begin{Proposition}[Duality]
  \label{prop:dual}
  There exists a unique bilinear map
  \[
  \langle \cdot, \cdot \rangle \colon W^{1-\alpha,1}((0, 1)^d) \times
  \rmsch^\gamma([0, 1]^d) \to \R
  \]
  such that, for all $f \in W^{1-\alpha,1}((0, 1)^d)$,
  \begin{itemize}
  \item[(A)] $\langle f, g \diff \calL\rangle = \int fg$ for $g \in
    L^\infty([0, 1]^d)$;
  \item[(B)] for all sequence $(\bomega_p)$ and $\bomega$ such that
    $\sup_p \|\bomega_p\|_\infty < \infty$ and $\|\bomega_p -
    \bomega\| \to 0$, one has $\langle f, \bomega_p\rangle \to \langle
    f, \bomega\rangle$.
  \end{itemize}
  Moreover, one has
  \begin{itemize}
  \item[(C)] if $f \in C^\beta([0, 1]^d)$ and $\beta + d\gamma > d$,
    then $\langle f, \bomega \rangle = \int_{[0, 1]^d} f \, \bomega$.
  \item[(D)] $\rmsch^\gamma([0, 1]^d)$ is the dual space of
    $W^{1-\alpha,1}((0, 1)^d)$ (under $\langle \cdot, \cdot \rangle$).
  \end{itemize}
\end{Proposition}

\begin{proof}
  The uniqueness follows from Corollary~\ref{cor:approx}. We turn to
  the existence. We set
  \begin{equation}
    \label{eq:defUps}
    \langle f, \bomega\rangle = \bomega([0, 1]^d) \int f +
    \sum_{n=0}^\infty \sum_{k=0}^{2^{nd}-1} \sum_{r=1}^{2^d-1} 2^{nd
      \left( \gamma - \frac{1}{2} \right)} b_{n,k,r}(f)
    a_{n,k,r}(\bomega)
  \end{equation}
  for all $f \in W^{1-\alpha,1}((0, 1)^d)$ and $\bomega \in
  \rmsch^\gamma([0, 1]^d)$. The above series converges by
  Lemma~\ref{lemma:decay} and the equivalence of the norms $\| \cdot
  \|_{W^{1-\alpha,1}}$ and $\bN^\alpha$. To prove (A), one notes that
  $f \in L^{1/\gamma}([0, 1]^d)$ by the fractional Gagliardo-Sobolev
  inequality, and therefore we can decompose $f$ along the Haar basis
  (which is indeed a basis of $L^{1/\gamma}$).
  \[
  f = \int f + \sum_{n=0}^\infty \sum_{k=0}^{2^{nd}-1}
  \sum_{r=1}^{2^d-1} 2^{nd\left( \gamma - \frac{1}{2} \right)}
  b_{n,k,r}(f) g_{n,k,r}
  \]
  Denoting $\bomega = g \diff \calL$, for $g \in L^\infty([0, 1]^d)$,
  \[
  g = \bomega([0, 1]^d) + \sum_{n=0}^\infty \sum_{k=0}^{2^{nd}-1}
  \sum_{r=1}^{2^d-1} a_{n,k,r}(\bomega) g_{n,k,r}
  \]
  is the Haar decomposition of $g$, that converges in $L^{1/(1 -
    \gamma)}([0, 1]^d)$. Consequently, the right-hand side
  of~\eqref{eq:defUps} is equal to $\int fg$. As for (B), we apply
  Lemma~\ref{lemma:decay} and the Lebesgue dominated convergence
  theorem. (C) is a consequence of~\eqref{eq:youngIntegral}.

  From preceding remarks, we know that $W^{1-\alpha,1}([0, 1]^d)$ is
  isomorphic to $\ell^1$ and $\rmsch^\gamma([0, 1]^d)$ to
  $\ell^\infty$. Up to these isomorphism, $\langle \cdot, \cdot
  \rangle$ is just the duality bracket between $\ell^1$ and
  $\ell^\infty$. This proves (D).
\end{proof}

\bibliographystyle{amsplain} \bibliography{phil.bib}

\end{document}